\documentclass[twoside]{article}

\usepackage{listings}

\usepackage{geometry}
\usepackage{scrhack}

\usepackage[utopia]{mathdesign}
\usepackage{dsfont, amsthm}

\newcommand{\SetFigFont}[2]{}

\usepackage[utf8]{inputenc}

\usepackage[american]{babel}

\usepackage{color}
\usepackage{array}
\usepackage{booktabs}
\usepackage[final]{graphicx}
\usepackage{enumitem}
\usepackage{url}
\usepackage{varwidth}
\usepackage{hhline}

\usepackage{amsmath}
\usepackage{mathtools}

\usepackage[
  normal,
  font={small,color=black},
  labelfont=bf,
  margin=2em,
  labelsep=period]{caption}[2008/04/01]

\usepackage{caption}

\usepackage[square,numbers,super,sort&compress]{natbib}

\makeatletter
\renewcommand\NAT@citesuper[3]{\ifNAT@swa
\if*#2*\else#2\NAT@spacechar\fi
\unskip\kern\p@\textsuperscript{\NAT@@open#1\if*#3*\else,\NAT@spacechar#3\fi\NAT@@close}%
   \else #1\fi\endgroup}
\makeatother
\usepackage{makeidx}         % allows index generation
\usepackage{graphicx}        % standard LaTeX graphics tool
\usepackage{multicol}        % used for the two-column index
\usepackage[bottom]{footmisc}% places footnotes at page bottom

\usepackage{array}

\providecommand{\keywords}[1]{\textbf{Keywords} #1}

\newtheorem{theorem}{Theorem}[section]
\newtheorem{definition}{Definition}[section]
\newtheorem{lemma}{Lemma}[section]

\usepackage{courier}

\usepackage{booktabs} % For formal tables

\usepackage[ruled]{algorithm2e} % For algorithms

\usepackage{fancyhdr}
\usepackage{amsmath}
\begin{document}
% Title portion
\title{On the implementation of a locally modified finite element method for interface problems in deal.II}

\author{
Stefan Frei\thanks{Department of Mathematics, University College London, UK (s.frei@ucl.ac.uk)}
\and Thomas Richter\thanks{Institut f\"ur Analysis und Numerik, Universit\"at Magdeburg, Germany 
(thomas.richter@ovgu.de)}, 
\and Thomas Wick\thanks{Institut f\"ur Angewandte Mathematik, Leibniz Universit\"at Hannover, 
Germany (thomas.wick@ifam.uni-hannover.de)}}

\date{}
\maketitle

% \author{Stefan Frei}
% %\orcid{1234-5678-9012-3456}
% \affiliation{%
%   \institution{Department of Mathematics, University College London}
%   \streetaddress{Gower Street}
%   \city{London}
% %  \state{VA}
%   \postcode{WC1E 6BT}
%   \country{United Kingdom}}
% \email{s.frei@ucl.ac.uk}
% %
% %
% %
% \author{Thomas Richter}
% \affiliation{%
%   \institution{Institut f\"ur Analysis und Numerik, Universit\"at Magdeburg}
%   \streetaddress{Universit\"atsplatz 2}
%   \city{Magdeburg}
% %  \state{VA}
%   \postcode{39106}
%   \country{Germany}}
% \email{thomas.richter@ovgu.de}
% %
% %
% %
% \author{Thomas Wick}
% \affiliation{%
%   \institution{Institut f\"ur Angewandte Mathematik, Leibniz Universit\"at Hannover}
%   \streetaddress{Welfengarten 1}
%   \city{Hannover}
% %  \state{VA}
%   \postcode{30167}
%   \country{Germany}}
% \email{thomas.wick@ifam.uni-hannover.de}

%\renewcommand\shortauthors{Frei, S. et al}

\begin{abstract}
In this work, we describe a simple finite element approach 
that is able to resolve weak discontinuities in interface problems accurately.
The approach is based on a fixed patch mesh consisting of quadrilaterals, 
that will stay unchanged independent of the position of the interface. Inside the
patches we refine once more, either in eight triangles or in four quadrilaterals, 
in such a way that the interface is locally resolved.
The resulting finite element approach can be considered a fitted 
finite element approach. In our practical implementation, we do not construct 
this fitted mesh, however. Instead, the local degrees of freedom are included 
in a parametric way in the finite element space, or to be more precise in the 
local mappings between a reference patch and the physical patches.
We describe the implementation in the open source C++ finite element 
library deal.II in detail and present two 
numerical examples to illustrate the performance of the approach. 
Finally, detailed studies of the behavior of iterative linear solvers
complement this work.
\end{abstract}


%\ccsdesc[300]{Mathematics of computing~Discretization}




\keywords{Locally modified finite elements; fitted finite elements; interface problems; C++; deal.II; implementation}

\maketitle

%%%%%%%%%%%%%%%%%%%%%%%%%%%%%%%%%%%%%%%%%%%%%%%%%%%%%%%%%%%%%%%%%%%%%%%%%%
\section{Introduction}

In this paper, we consider interface problems, where the solution is continuous on a
domain  $\Omega \subset \mathbb{R}^2$, but its normal derivative may have a jump in normal
direction over an interior interface. Problems of this kind 
arise for example in fluid-structure interaction, 
multiphase flows, multicomponent structures and in many other
configurations where multiple physical phenomena interact.
All these examples have in common that the interface between
the two phases is moving and may be difficult to capture due to small
scale features. 

%One approach to deal with such problems is to move the mesh with the interface, using
%so-called \textit{interface-tracking} algorithms as the Arbitrary Lagrangian
%Eulerian method \cite{HiAmCo74,HuLiZi81,DoGiuHa82,FoNo99}
%In the case of large interface movements, such an approach is however challenging, as the 
%resulting finite element meshes might become highly distorted or even degenerate. Moreover,
%topology changes in the sub-domains are not possible with this approach.

%An alternative consists of so-called \textit{interface-capturing} techniques, where the interface 
%moves over a typically fixed finite element mesh. 

{ If the interface is not resolved by the finite element mesh,}
the accuracy of the finite element
approach might decrease severely, see e.g.~\cite{Babuska1970}. 
For simple elliptic interface problem with 
jumping coefficients, it has been shown, that
optimal convergence can be
recovered by a harmonic averaging of the diffusion
constants~\cite{TikhonovSamarskii1962}, \cite{ShubinBell1984}. 

For more complex couplings, e.g.$\,$ fluid-structure interactions, where two entirely different
equations interact with each other, the list of possible { discretisation techniques} that yield optimal order
can be split roughly in two groups.

The first class of approaches consists of so-called fitted finite element methods, where 
the meshes are constructed in such a way that the interface is sufficiently
resolved, see~\cite{Babuska1970, FeistauerSobotikova1990,
  Zenisek1990,   BrambleKing1996, BastingPrignitz2013}. 
If the interface is moving,
curved or has small scale features, the 
repeated generation of fitted finite element meshes
can exceed the feasible effort, however. In non-stationary problems, the projection of
previous iterates to the new mesh, brings along further difficulties and sources of error.
Further developments are based on local modifications
of the finite element mesh, that only alter mesh elements close to the
interface~\cite{Boergers1990, XieItoLiToivanen2008, Fang2013, GawlikLew2014}. 

An alternative approach is based on unfitted finite elements,
where the mesh is fixed and does not resolve the interface. Here,
proper accuracy is gained by local modifications or enrichment of the
finite element basis. Prominent examples for these methods are
the extended finite element method
(XFEM~\cite{MoesDolbowBelytschko1999}), the generalised finite element  
method~\cite{BabuskaBanarjeeOsborn2004,BaBa12} or the unfitted Nitsche
method by \cite{HansboHansbo2002, HansboHansbo2004}. 
Based on the latter works, so-called cut
finite elements have been developed,
see for instance~\cite{BastianEngwer}, \cite{Burmanetal2015}, 
\cite{Lehrenfeld4d}, \cite{FidkowskiDarmofal2007}.
All these enrichment 
methods are well analysed and show the correct order of
convergence. One drawback of the enrichment methods is a complicated structure
that requires local modifications in the finite element spaces leading
to a variation in the connectivity of the system matrix and number of
unknowns.

In this article, we use a simple approach that is based on a fixed \textit{patch
mesh} consisting of quadrilaterals and  will stay unchanged independent of the position of the interface. Inside the
patches we refine once more, either in eight triangles or in four quadrilaterals, in such a way that the interface is locally resolved.
In this sense the resulting finite element approach can be considered a \textit{fitted} finite element approach. 
{This approach has first been proposed in~\cite{FreiRichter2014}.}
In our practical
implementation, we do however not construct this fitted mesh explicitly. 
Instead, the local degrees of freedom are included in a parametric
way in the finite element space, or to be more precise in the local mappings between a reference patch and the physical patches.
 
The drawback of this approach is that the condition number of the resulting system matrices might be unbounded, when the 
interface approaches certain vertices or mesh lines. This problem can however be solved by constructing
a scaled hierarchical basis of the finite element space. Using this basis the approach can be viewed as a simple 
enrichment method as well, where the enrichment consists of the standard Lagrangian basis functions on the fine scale.
 
The mathematical details, including a complete analysis of the discretisation error and the condition number 
of the system matrix
have already been published in~\cite{FreiRichter2014}, \cite{FreiDiss} and~\cite{RichterBuch}.
Later on, related approaches on triangular patches have been developed by~\cite{HolmFSIband} and by~\cite{GanglLanger}. 
Furthermore, the approach has been applied by the authors to simulate fluid-structure interaction problems with 
large deformations
in~\cite{FrRiWi14b,FreiRichterWick2016} and~\cite{FreiRichterSammelband} and by Gangl to simulate 
problems of topology optimisation~\cite{GanglDiss}.

The goal of this article is to explain in detail the implementation of the 
fitted finite element method and to provide a programming code 
based on the C++ finite element library deal.II~\cite{BangerthHartmannKanschat2007,dealII85}.
In {extension} to \cite{FreiRichter2014} further details concerning the implementation of 
the finite element approach, and in particular on the construction of the hierarchical basis are given.
Moreover, we study the performance of some iterative solvers, i.e.$\,$ a simple and preconditioned
conjugate gradient method (CG/PCG) to solve the arising linear systems, while
{`only'} a direct solver was used in \cite{FreiRichter2014}.
% From the numerical solver side, we extend in this work the previously mentioned 
% studies from a direct solver 
% to conjugate gradients (CG), which is a well-known iterative solver. 
% Moreover, we perform tests using symmetric successive 
% over-relaxation (SSOR) as preconditioner for the CG solver.

The organisation of this article is as follows. 
In Section \ref{sec_model_problem}, {a simple elliptic model problem} 
is presented. Next, in Section \ref{sec:fe} {we introduce the local modifications
of the finite element space in the cells that are cut by the interface.}
In Section \ref{sec_discrete_forms} the discrete forms and 
the approximation properties are briefly recapitulated.
Then, we introduce the hierarchical 
finite element space in Section~\ref{sec_hierarchical}.
Section \ref{sec:num} consists of two numerical tests, that illustrate the main 
features and the performance of our approach.
Finally, we present algorithmic details and details on the 
implementation in Section~\ref{sec.impl}.
We conclude in Section~\ref{sec.conclusion}. 
% we provide three appendices. In Appendix $A$, the installation 
% and compilation of the code are explained. In Appendix $B$, 
% we outline the basic structure of the code. Finally, in 
% Appendix $C$ details of the implementation are shown by means of code fragments, 
% including a detailed explanation of how to use the functions that are specific to the 
% locally modified finite element method.

%%%%%%%%%%%%%%%%%%%%%%%%%%%%%%%%%%%%%%%%%%%%%%%%%%%%%%%%%%%%%%%%%%%%%%%%%
\section{Motivation: A simple elliptic model problem}
\label{sec_model_problem}
To get started, let us consider a simple
Poisson problem in $\Omega\subset\mathbb{R}^2$ 
with a discontinuous coefficient $\kappa$ across an interface line
$\Gamma\subset\mathbb{R}$. Find $u:\Omega\to\mathbb{R}$ such that

\begin{equation}\label{problem:1}
  -\nabla\cdot (\kappa_i\nabla u) = f\text{ on
  }\Omega_i\;\; (i=1,2),\quad
  [u]=0 \text{ and } [\kappa\partial_n u] = 0\text{ on }\Gamma, 
\end{equation}
with constants $\kappa_i>0$ and subject to homogeneous Dirichlet conditions on the 
exterior boundary $\partial\Omega$. Here, we denote 
the subdomains by $\Omega_i, i=1,2$ and by $[u]$
% \[
% [u](x):=\lim_{s\downarrow 0}u(x+sn)-\lim_{s\uparrow 0}u(x+sn),\quad
% x\in\Gamma, 
% \]
the jump of $u$ across the interface $\Gamma$. The variational
formulation of this interface problem is given by
\begin{definition}[Continuous variational formulation]
Find $u\in H^1_0(\Omega)$ such that 
\begin{align}
a(u,\phi) :=\sum_{i=1}^2 (\kappa_i\nabla
u,\nabla\phi) = (f,\phi)\quad\forall\phi\in H^1_0(\Omega).\label{contBilin}
\end{align}
\end{definition}

Interface problems are elaborately discussed in literature. If the
interface $\Gamma$ cannot be resolved by the mesh, the overall 
error for a standard finite element approach
will be bounded by
\[
\|\nabla (u-u_h)\|_\Omega = \mathcal{O}(h^{1/2}),
\]
independent of the polynomial degree $r$ of the finite element
space, see the early works \cite{Babuska1970} or~\cite{MacKinnonCarey1987}. In Figure~\ref{fig:standardfe},
we show the $H^1$- and $L^2$-norm errors for a simple interface problem with a
curved interface that is not resolved by the finite element mesh. 
Both linear and quadratic finite elements yield only $\mathcal{O}(h^{1/2})$
accuracy in the $H^1$-semi-norm and $\mathcal{O}(h)$ in the $L^2$-norm. This is
due to the limited regularity of the solution across the interface.  

\begin{figure}[t]
  \begin{minipage}{0.63\textwidth}
  \hspace{-0.8cm}
    \includegraphics[width=\textwidth]{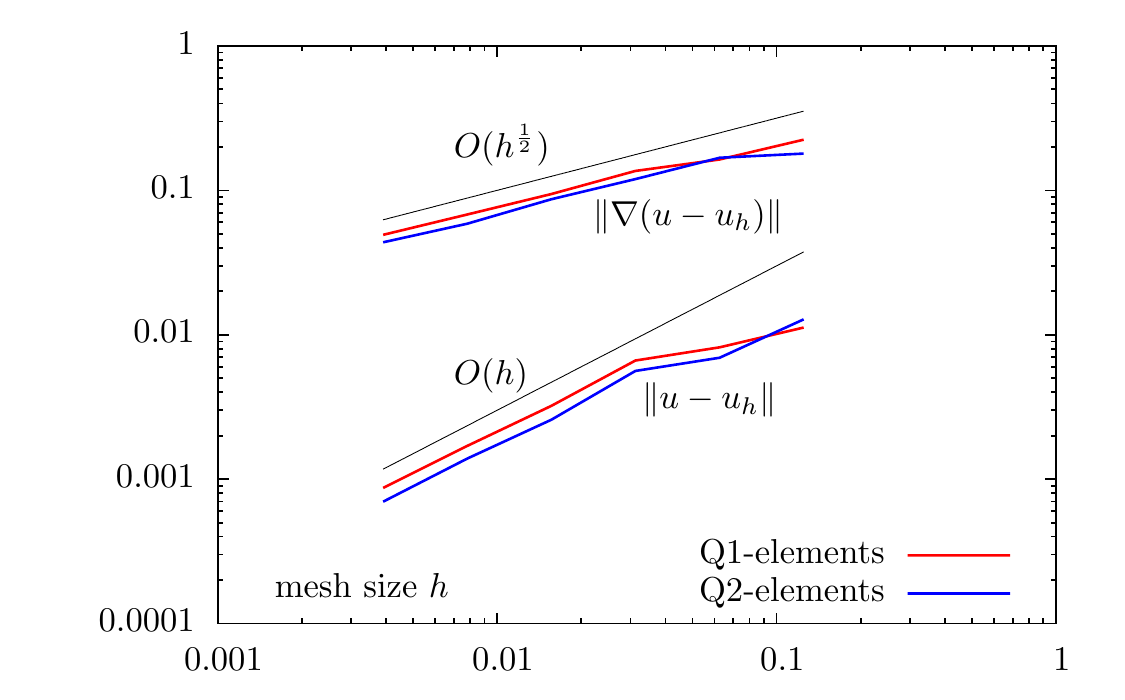}
  \end{minipage}
  \hfil
  \begin{minipage}{0.33\textwidth}
  %\hspace{-5cm}
  
  \begin{picture}(0,0)%
\includegraphics{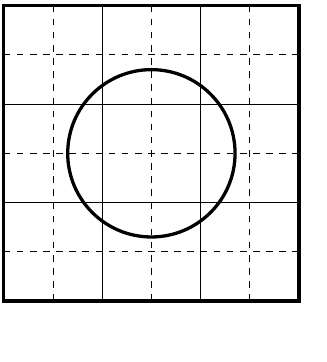}%
\end{picture}%
\setlength{\unitlength}{2072sp}%
\begingroup\makeatletter\ifx\SetFigFont\undefined%
\gdef\SetFigFont#1#2{%
  \fontsize{#1}{#2pt}%
  \selectfont}%
\fi\endgroup%
\begin{picture}(2973,3283)(2218,-3761)
\put(2746,-3070){\makebox(0,0)[lb]{\smash{{\SetFigFont{8}{9.6}{\color[rgb]{0,0,0}$\Omega_2$}%
}}}}
\put(5176,-1411){\makebox(0,0)[lb]{\smash{{\SetFigFont{9}{10.8}{\color[rgb]{0,0,0}$-\kappa_i\Delta u=f$}%
}}}}
\put(5176,-2086){\makebox(0,0)[lb]{\smash{{\SetFigFont{9}{10.8}{\color[rgb]{0,0,0}$u=u^d$ on $\partial\Omega$}%
}}}}
\put(2251,-3661){\makebox(0,0)[lb]{\smash{{\SetFigFont{9}{10.8}{\color[rgb]{0,0,0}$\kappa_1=0.1,\; \kappa_2=1$}%
}}}}
\put(3196,-1750){\makebox(0,0)[lb]{\smash{{\SetFigFont{8}{9.6}{\color[rgb]{0,0,0}$\Omega_1$}%
}}}}
\end{picture}%
    
  \end{minipage}
  \caption{$L^2$- and $H^1$-error for a standard finite element
     method using $Q_1$ and $Q_2$ polynomials for the discretisation 
     of the interface problem~\eqref{problem:1}. Configuration of the test problem in the
     right sketch. Further details are given in
     Section~\ref{sec:num}.\label{fig:standardfe}}  
 \end{figure}
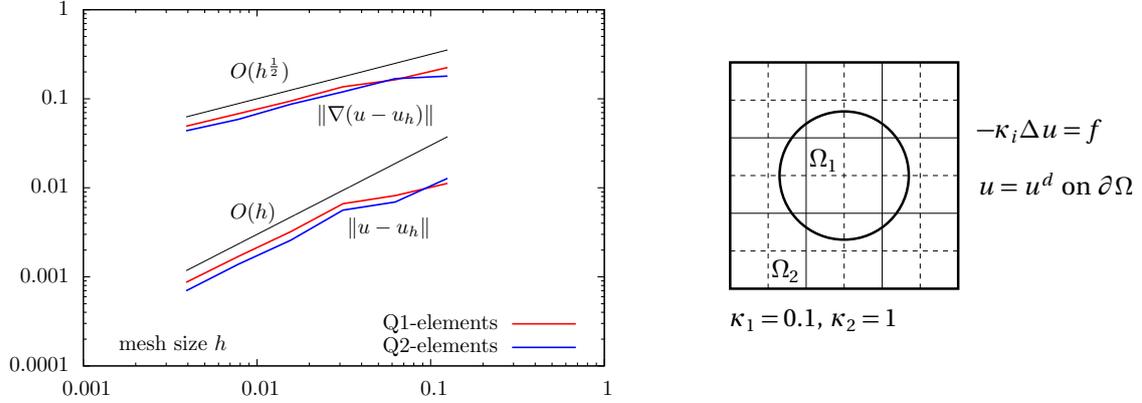

%%%%%%%%%%%%%%%%%%%%%%%%%%%%%%%%%%%%%%%%%%%%%%%%%%%%%%%%%%%%%%%%%%%%%%%%%
\section{Locally modified finite elements}
\label{sec:fe}

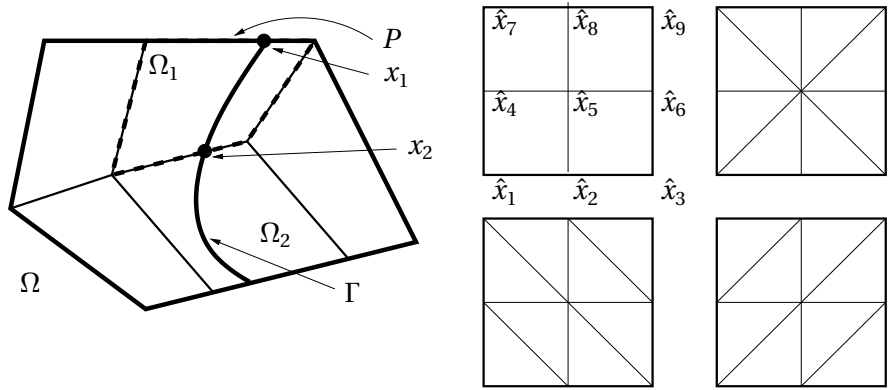
\begin{figure}[t]
  \centering
  \resizebox*{0.8\textwidth}{!}{
  
  \begin{picture}(0,0)%
\includegraphics{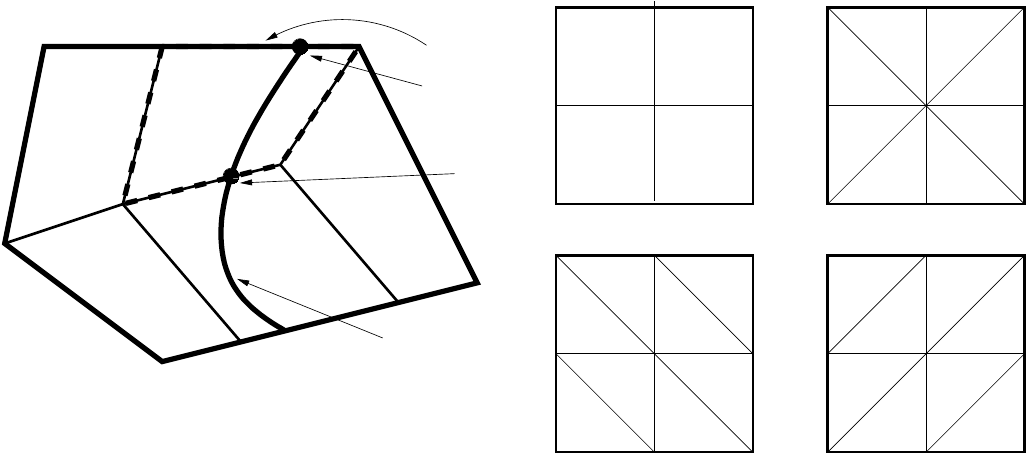}%
\end{picture}%
\setlength{\unitlength}{1657sp}%
\begingroup\makeatletter\ifx\SetFigFont\undefined%
\gdef\SetFigFont#1#2{%
  \fontsize{#1}{#2pt}%
  \selectfont}%
\fi\endgroup%
\begin{picture}(11742,5197)(1747,-5179)
\put(8191,-2634){\makebox(0,0)[lb]{\smash{{\SetFigFont{5}{6.0}{\color[rgb]{0,0,0}$\hat x_1$}%
}}}}
\put(9271,-2634){\makebox(0,0)[lb]{\smash{{\SetFigFont{5}{6.0}{\color[rgb]{0,0,0}$\hat x_2$}%
}}}}
\put(10441,-2634){\makebox(0,0)[lb]{\smash{{\SetFigFont{5}{6.0}{\color[rgb]{0,0,0}$\hat x_3$}%
}}}}
\put(8191,-1451){\makebox(0,0)[lb]{\smash{{\SetFigFont{5}{6.0}{\color[rgb]{0,0,0}$\hat x_4$}%
}}}}
\put(9271,-1451){\makebox(0,0)[lb]{\smash{{\SetFigFont{5}{6.0}{\color[rgb]{0,0,0}$\hat x_5$}%
}}}}
\put(10441,-1451){\makebox(0,0)[lb]{\smash{{\SetFigFont{5}{6.0}{\color[rgb]{0,0,0}$\hat x_6$}%
}}}}
\put(8191,-358){\makebox(0,0)[lb]{\smash{{\SetFigFont{5}{6.0}{\color[rgb]{0,0,0}$\hat x_7$}%
}}}}
\put(9271,-358){\makebox(0,0)[lb]{\smash{{\SetFigFont{5}{6.0}{\color[rgb]{0,0,0}$\hat x_8$}%
}}}}
\put(10441,-358){\makebox(0,0)[lb]{\smash{{\SetFigFont{5}{6.0}{\color[rgb]{0,0,0}$\hat x_9$}%
}}}}
\put(3646,-961){\makebox(0,0)[lb]{\smash{{\SetFigFont{7}{8.4}{\color[rgb]{0,0,0}$\Omega_1$}%
}}}}
\put(5131,-3211){\makebox(0,0)[lb]{\smash{{\SetFigFont{7}{8.4}{\color[rgb]{0,0,0}$\Omega_2$}%
}}}}
\put(1936,-3886){\makebox(0,0)[lb]{\smash{{\SetFigFont{7}{8.4}{\color[rgb]{0,0,0}$\Omega$}%
}}}}
\put(6751,-601){\makebox(0,0)[lb]{\smash{{\SetFigFont{8}{9.6}{\color[rgb]{0,0,0}$P$}%
}}}}
\put(6256,-4021){\makebox(0,0)[lb]{\smash{{\SetFigFont{7}{8.4}{\color[rgb]{0,0,0}$\Gamma$}%
}}}}
\put(7071,-2025){\makebox(0,0)[lb]{\smash{{\SetFigFont{6}{7.2}{\color[rgb]{0,0,0}$x_2$}%
}}}}
\put(6751,-1051){\makebox(0,0)[lb]{\smash{{\SetFigFont{6}{7.2}{\color[rgb]{0,0,0}$x_1$}%
}}}}
\end{picture}%

  }

  \caption{\textit{Left:} Triangulation $\mathcal{T}_{2h}$ of a domain $\Omega$ that
    is split into $\Omega_1$ and $\Omega_2$ with interface
    $\Gamma$. The quadrilateral cells in $\mathcal{T}_{2h}$ are illustrated by the bold lines. 
    Patch $P$ is cut by $\Gamma$ at $x_1$ and $x_2$. \textit{Right:}
    Subdivision of reference patches $\hat P_0,\hat P_1,\hat P_2,\hat
    P_3$ (top left to bottom right) into eight triangles each.
    } 
  \label{fig:mesh}
\end{figure}

In order to define the modified finite elements, let us assume that 
$\mathcal{T}_{2h}$ is a form and shape-regular triangulation of the domain
$\Omega\subset\mathbb{R}^2$ into open quadrilaterals. The discrete domain $\Omega_h$
does not necessarily resolve the partitioning
$\Omega=\Omega_1\cup\Gamma\cup\Omega_2$ and the interface $\Gamma$ can
cut the elements $P\in\mathcal{T}_{2h}$. 

We assume that the interface $\Gamma$ cuts patches in the following way: 
\begin{enumerate}
\item Each (open) patch $P\in\mathcal{T}_{2h}$ is either not cut
  $P\cap\Gamma=\emptyset$ or cut in exactly two points on its 
  boundary: $P\cap\Gamma\neq \emptyset$ and $\partial
  P\cap\Gamma=\{x^P_1,x^P_2\}$.
\item If a patch is cut, the two cut-points $x^P_1$ and $x^P_2$ may
  not be inner points of the same edge.
\end{enumerate}
In principle, these assumptions only rule out two possibilities: a
patch may not be cut multiple times and the interface may not enter
and leave the patch at the same edge. Both situations can be avoided
by refinement of the underlying mesh. If the interface is matched by
an edge, the patch is not considered to be cut.

\subsection{Construction of the finite element space}

We define four reference patches $\hat{P}_0,...,\hat{P}_3$ on the unit square
$(0,1)^2$. These patches are split into 4 quadrilaterals or 8 triangles
as illustrated in Figure~\ref{fig:mesh}.
Moreover, we define 9 nodes $\hat x_1,\dots,\hat x_9$ in the vertices, edge midpoints and the midpoint
of the patches, which will serve as degrees of freedom of the finite element space. Note that the
same position of the degrees of freedom can be found in a standard quadratic $Q_2$ discretisation,
the structure of which served as a starting point for our implementation. 

Now we define local reference spaces $\hat Q_P$ (here $P$ indicates the
  patch, but not the polynomial degree) as a piecewise polynomial space of degree 1.
On the reference patch $\hat P_0$ consisting of quadrilaterals $\hat K_1,\dots, \hat K_4$,
we choose the standard space
of piecewise bilinear functions 
\[
\hat Q_P = \hat Q := \left\{ \phi\in C(\bar P),\; \phi\Big|_{\hat K_i}\in
\operatorname{span}\{1,x,y,xy\},\; \hat K_1,\dots,\hat K_4\in \hat P\right\}.  
\]
This local space will be used when a physical patch $P$ is not cut by the interface.
If a patch $P\in\mathcal{T}_{2h}$ is cut by the interface, we use one of the 
reference patches $\hat P_1, \dots, \hat P_3$ with triangles $\hat T_1,\dots,\hat T_8$ and define 
\[
\hat Q_P = \hat Q_\text{mod} := \left\{ \phi\in C(\bar P),\; \phi\Big|_{\hat T_i}\in
\operatorname{span}\{1,x,y\},\; \hat T_1,\dots,\hat T_8\in \hat P\right\}. 
\]

We define a mapping $\hat T_P \in \hat Q_P, \hat T_P: \hat{P}_i \to P$, that is piecewise linear in sub-triangles
and piecewise bi-linear in sub-quadrilaterals on $\hat{P}_i$. This gives us the possibility to map the degrees of freedom
$\hat x_1,\dots, \hat x_9$ to nodes $x_1^P,\dots,x_9^P$, in such a way that the interface is resolved in a 
linear approximation in the physical patch.
% \[
% T_P(\hat x_i) = x_i^P,\quad i=1,\dots,9,
% \]
Denoting by 
$\{\hat\phi^1,\dots,\hat\phi^9\}$ the standard Lagrange basis of
$\hat Q$ or $\hat Q_\text{mod}$ with
$\hat\phi^i(\hat{x}_j)=\delta_{ij}$, the transformation $\hat{T}_P$ is
given by
\begin{equation}
\hat{T}_P(\hat{x})=\sum_{i=1}^9 x_i^P\hat\phi_i(\hat{x}). \label{TP}
\end{equation}

Finally, we define the finite element trial space $V_h\subset H^1_0(\Omega)$ as an
iso-parametric space on the triangulation $\mathcal{T}_{2h}$:
\[
V_h = \left\{\phi\in C(\bar\Omega) \cap H^1_0(\Omega),\; \phi\circ \hat{T}_P^{-1}\Big|_P\in
  \hat Q_P\text{ for all patches }P\in\mathcal{T}_{2h}\right\}.
\]
% where $T_P\in [\hat Q_P]^2$ is the mapping between the reference patch
% $\hat P=(0,1)^2$ and the patch $P\in\mathcal{T}_{2h}$ such that
% \[
% T_P(\hat x_i) = x_i^P,\quad i=1,\dots,9,
% \]
% for the nine nodes $x_1^P,\dots,x_9^P$ of the patch, see
% Figure~\ref{fig:mesh}. 
% Depending on the position of the interface $\Gamma$ in the patch $P$,
% three different reference configurations are
% considered, see the right sketch in Figure~\ref{fig:mesh}.

Note that, whatever splitting of the patch is applied, the local number of degrees
of freedom is always 9. Therefore, the global number of unknowns and the {sparsity pattern} of the system matrix 
stays identical, independent of the interface position.

It is important to note, that the functions in $\hat Q$ and $\hat
Q_\text{mod}$ are all piecewise linear on the edges $\partial P$, such
that  mixing different element types does not affect the continuity of
the global finite element space. 

% \begin{figure}[t]
%   \centering
%   \input{patches}
%   \caption{Discretization of patches cut by the interface into a
%     subdivision of triangles.}\label{fig:patches}
% \end{figure}
% 
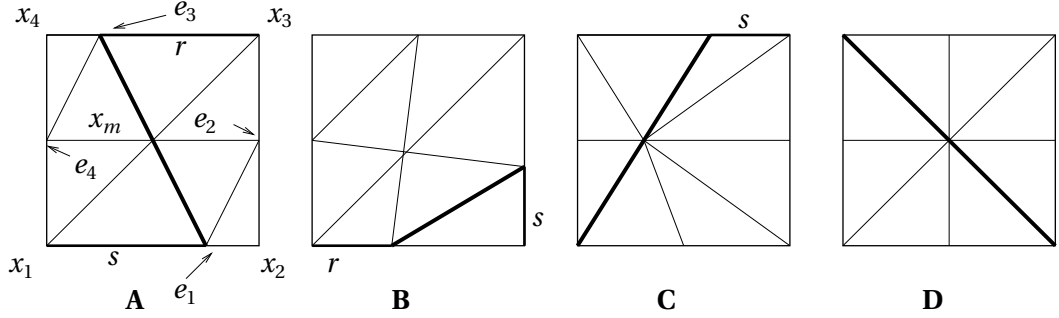
\begin{figure}[t]
  \centering
\scalebox{1.1}{
\begin{picture}(0,0)%
\includegraphics{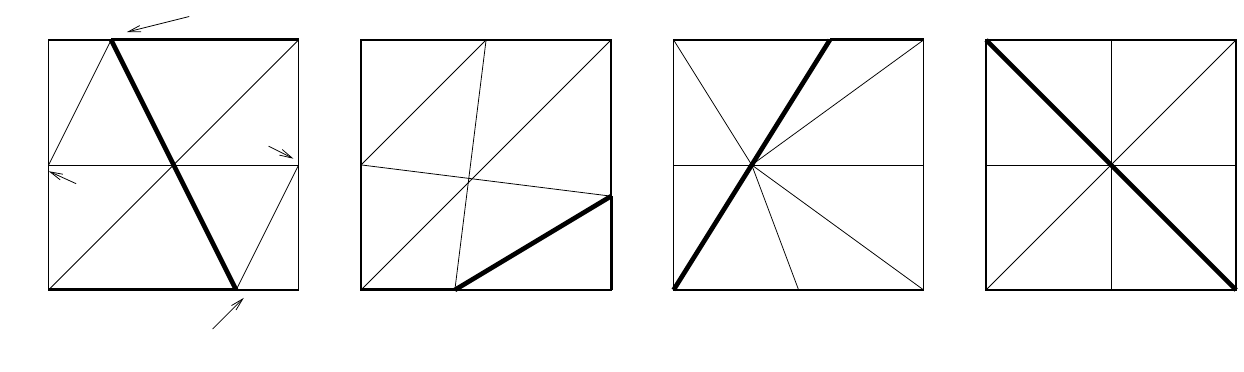}%
\end{picture}%
\setlength{\unitlength}{1973sp}%
\begingroup\makeatletter\ifx\SetFigFont\undefined%
\gdef\SetFigFont#1#2{%
  \fontsize{#1}{#2pt}%
  \selectfont}%
\fi\endgroup%
\begin{picture}(11909,3544)(736,-3545)
\put(2626,-136){\makebox(0,0)[lb]{\smash{{\SetFigFont{7}{8.4}{\color[rgb]{0,0,0}$e_3$}%
}}}}
\put(4351,-3061){\makebox(0,0)[lb]{\smash{{\SetFigFont{7}{8.4}{\color[rgb]{0,0,0}$r$}%
}}}}
\put(1651,-1411){\makebox(0,0)[lb]{\smash{{\SetFigFont{7}{8.4}{\color[rgb]{0,0,0}$x_m$}%
}}}}
\put(751,-3061){\makebox(0,0)[lb]{\smash{{\SetFigFont{7}{8.4}{\color[rgb]{0,0,0}$x_1$}%
}}}}
\put(3676,-211){\makebox(0,0)[lb]{\smash{{\SetFigFont{7}{8.4}{\color[rgb]{0,0,0}$x_3$}%
}}}}
\put(826,-211){\makebox(0,0)[lb]{\smash{{\SetFigFont{7}{8.4}{\color[rgb]{0,0,0}$x_4$}%
}}}}
\put(2626,-3361){\makebox(0,0)[lb]{\smash{{\SetFigFont{7}{8.4}{\color[rgb]{0,0,0}$e_1$}%
}}}}
\put(3601,-3061){\makebox(0,0)[lb]{\smash{{\SetFigFont{7}{8.4}{\color[rgb]{0,0,0}$x_2$}%
}}}}
\put(1501,-1936){\makebox(0,0)[lb]{\smash{{\SetFigFont{7}{8.4}{\color[rgb]{0,0,0}$e_4$}%
}}}}
\put(1876,-2986){\makebox(0,0)[lb]{\smash{{\SetFigFont{7}{8.4}{\color[rgb]{0,0,0}$s$}%
}}}}
\put(2851,-1411){\makebox(0,0)[lb]{\smash{{\SetFigFont{7}{8.4}{\color[rgb]{0,0,0}$e_2$}%
}}}}
\put(6676,-2536){\makebox(0,0)[lb]{\smash{{\SetFigFont{7}{8.4}{\color[rgb]{0,0,0}$s$}%
}}}}
\put(2626,-586){\makebox(0,0)[lb]{\smash{{\SetFigFont{7}{8.4}{\color[rgb]{0,0,0}$r$}%
}}}}
\put(9001,-286){\makebox(0,0)[lb]{\smash{{\SetFigFont{7}{8.4}{\color[rgb]{0,0,0}$s$}%
}}}}
\put(2101,-3481){\makebox(0,0)[lb]{\smash{{\SetFigFont{6}{7.2}{\color[rgb]{0,0,0}$\textbf{A}$}%
}}}}
\put(5101,-3481){\makebox(0,0)[lb]{\smash{{\SetFigFont{6}{7.2}{\color[rgb]{0,0,0}$\textbf{B}$}%
}}}}
\put(8101,-3481){\makebox(0,0)[lb]{\smash{{\SetFigFont{6}{7.2}{\color[rgb]{0,0,0}$\textbf{C}$}%
}}}}
\put(11101,-3481){\makebox(0,0)[lb]{\smash{{\SetFigFont{6}{7.2}{\color[rgb]{0,0,0}$\textbf{D}$}%
}}}}
\end{picture}%

}
  \caption{Different types of cut patches. The subdivision can be
    anisotropic with $r,s\in (0,1)$ arbitrary.}\label{fig:types}
\end{figure}

Next, we present the subdivision of interface patches $P$ into eight
triangles.  

\begin{definition}
We distinguish four different types of interface
cuts, see Figure~\ref{fig:types}: 
\begin{description}
\item[Configuration A] The patch is cut in the interior of two
  opposite edges. 
  %We denote the two intersections by $x_1,x_2\in\partial  P$.
\item[Configuration B] The patch is cut in the interior of two
  adjacent edges.
  %, again denoted by $x_1,x_2\in\partial P$. 
\item[Configuration C] The patch is cut in the interior of one 
  edge and in one node. 
\item[Configuration D] The patch is cut in two opposite nodes. 
\end{description}
\end{definition}

Configurations A and B are based on the reference patches $\hat P_2$
and $\hat P_3$, configurations C and D use the reference patch $\hat
P_1$, see Figure~\ref{fig:mesh}.  

By $e_i\in\mathbb{R}^2$, $i=1,2,3,4$ we
denote the vertices in the interior of edges, by $m_P\in\mathbb{R}^2$ the grid point
in the interior of the patch. The parameters $r,s\in (0,1)$ describe the relative
position of the intersection points with the interface on the outer
edges. 

If an edge is intersected by the interface, we move the corresponding point 
$e_i$ on this edge to the point of intersection. The position of $m_P$ 
depends on the specific
configuration. For configuration A, B and D, we choose $m_P$ as the
intersection of the line connecting $e_2$ and $e_4$ with the line connecting
$e_1$ and $e_3$. In configuration C, we use the intersection of the line
connecting $e_2$ and $e_4$ with
the line connecting $x_1$ and $e_3$. 

As the cut of the elements can be arbitrary with $r,s\to 0$ or $r,s\to
1$, the triangle's aspect ratio can be very large. With the described choices for the midpoints $m_P$
we can guarantee, that the 
maximum angles in all triangles will be well bounded away from
$180^\circ$~\cite{FreiRichter2014}:
\begin{lemma}[Maximum angle condition]\label{lemma:maxangle}
  All interior angles of the triangles shown in
  Figure~\ref{fig:types} are bounded by $144^\circ$ independent of
  $r,s\in (0,1)$.
\end{lemma}

The respective reference patches $\hat P_0,...,\hat P_3$ (see Figure~\ref{fig:mesh})
are chosen based on the following criteria: First, it is mandatory that a maximum angle 
can be guaranteed. Second, it is beneficial for practical purposes
% While there is only one possibility to choose a reference patch in configurations B to D, 
% in configuration A both patches $\hat{P}_2$ and $\hat{P}_3$ could be chosen in principle.
% In both cases, the maximum angle condition is ensured. For practical purposes, 
%it is beneficial
to keep the maximum angle as small as possible on the one hand and on the other hand 
to conserve the symmetry in the discretisation, in the case of a symmetric problem. From these considerations, 
we choose type $\hat{P}_2$ if $r+s > 1$ and $\hat{P}_3$ if $r+s < 1$ for configuration A in our implementation. 
For an example, consider
the left patch in Figure~\ref{fig:types}, where $\hat{P}_3$ has been chosen, as $r+s>1$. {Note that the symmetry criterion 
would not be fulfilled, if we would choose always either $\hat{P}_2$ or $\hat{P}_3$, independent of $r$ and $s$.}
In configuration B, we 
choose $\hat{P}_3$, when the cut separates the lower left or the upper right vertex from the rest of the patch
and $\hat{P}_2$, when only the lower right or the upper left vertex lie on one side of the interface.

%%%%%%%%%%%%%%%%%%%%%%%%%%%%%%%%%%%%%%%%%%%%%%%%%%%%%%%%%%%%%%%%%%%%%%%%%
\section{Discrete variational formulation and approximation properties}
\label{sec_discrete_forms}
In the previous sections, we tacitly assumed that the 
interface can be resolved in a geometric exact way.
In the case of a curved interface, a linear 
approximation by mesh lines is constructed.

With the help of the discrete approximation of the interface, we 
introduce a second splitting of the domain $\Omega$ into
the discrete subdomains
\begin{align*}
 \Omega =  \Omega_h^1 \cup \Omega_h^2,
\end{align*}
such that all cells of the sub-triangulation are either completely included in $\Omega_h^1$ or 
in $\Omega_h^2$, see Figure~\ref{fig:discmesh}.

\begin{figure}[bt]
  \centering
\scalebox{1.2}{
\begin{picture}(0,0)%
\includegraphics{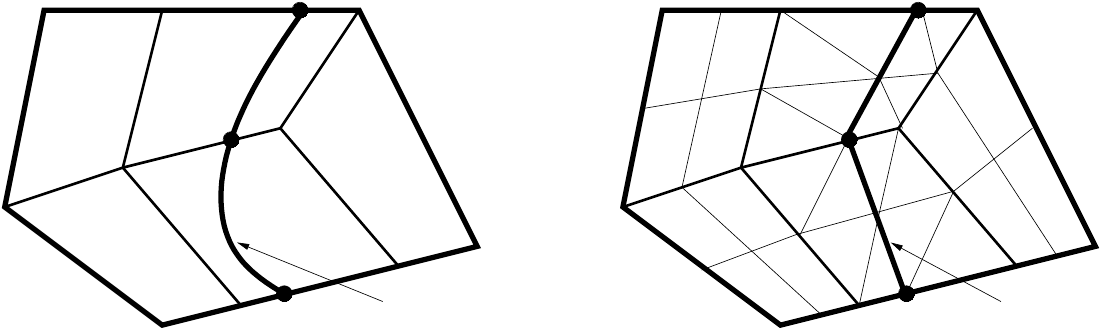}%
\end{picture}%
\setlength{\unitlength}{1657sp}%
\begingroup\makeatletter\ifx\SetFigFont\undefined%
\gdef\SetFigFont#1#2{%
  \fontsize{#1}{#2pt}%
  \selectfont}%
\fi\endgroup%
\begin{picture}(12573,3753)(1747,-4165)
\put(13321,-4021){\makebox(0,0)[lb]{\smash{{\SetFigFont{7}{8.4}{\color[rgb]{0,0,0}$\Gamma_h$}%
}}}}
\put(12196,-3211){\makebox(0,0)[lb]{\smash{{\SetFigFont{7}{8.4}{\color[rgb]{0,0,0}$\Omega_h^2$}%
}}}}
\put(10711,-961){\makebox(0,0)[lb]{\smash{{\SetFigFont{7}{8.4}{\color[rgb]{0,0,0}$\Omega_h^1$}%
}}}}
\put(3646,-961){\makebox(0,0)[lb]{\smash{{\SetFigFont{7}{8.4}{\color[rgb]{0,0,0}$\Omega_1$}%
}}}}
\put(5131,-3211){\makebox(0,0)[lb]{\smash{{\SetFigFont{7}{8.4}{\color[rgb]{0,0,0}$\Omega_2$}%
}}}}
\put(6256,-4021){\makebox(0,0)[lb]{\smash{{\SetFigFont{7}{8.4}{\color[rgb]{0,0,0}$\Gamma$}%
}}}}
\end{picture}%

}
  \caption{\label{fig:discmesh}
  {\textit{Left:} Patch elements and interface $\Gamma$ that goes through two patches.
  \textit{Right:} Sub-triangulation} and splitting of the mesh into subdomains $\Omega_h^1$ and $\Omega_h^2$.
  The interface $\Gamma_h$ is a linear approximation of the interface $\Gamma$ shown on the left-hand side.}
\end{figure}
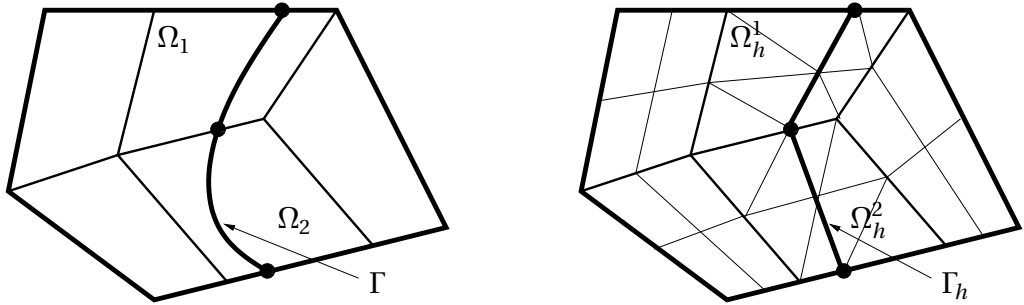

Using these definitions, we define a discrete bilinear form
$a_h(\cdot,\cdot)$. 
For the elliptic model problem, this form is given by
\begin{align}
 a_h(u_h,\phi_h) := (\kappa_h \nabla u_h,\nabla \phi_h)_{\Omega_h},\label{discBilin}
\end{align}
where 
\begin{align*}
\kappa_h = \begin{cases}
            \kappa_1 \quad \text{in } \Omega_h^1,\\
            \kappa_2 \quad \text{in } \Omega_h^2.\\
           \end{cases}
\end{align*}
Note that $\kappa_h$ differs from $\kappa$ in a small layer between
the continuous interface $\Gamma$ and the discrete interface $\Gamma_h$.

\begin{definition}[Discrete variational formulation]
The discrete problem is to find $u_h \in V_h$ such that 
\begin{align*}
 a_h(u_h,\phi_h) = (f,\phi_h)_{\Omega_h} \quad \forall \phi_h \in V_h. 
\end{align*}
% Together with (\ref{contBilin}), we have the Galerkin orthogonality
% \begin{align}
%  a(u,\phi_h)-a_h(u_h,\phi_h) = 0 \quad \forall \phi_h \in V_h. \label{GalOrth}
% \end{align}
\end{definition}

The maximum angle conditions of Lemma~\ref{lemma:maxangle}
is sufficient to ensure that the Lagrangian interpolation operators $I_h:H^2(T)\cap
C(\bar T)\to V_h$
are of optimal order for smooth functions $v\in H^2(T)\cap C(\bar T)$ on an element $T$, i.e.$\,$
\begin{equation}\label{interpolation}
  \|\nabla^k (v-I_h v)\|_T \le c h_{T,\max}^{2-k} \|\nabla^2
  v\|_T,\quad  k=0,1
\end{equation}
where $c>0$ is a constant and $h_{T,\max}$ is the maximum diameter
of a triangle $T\in P$ (see e.g.~\cite{Apel1999}). If the interface $\Gamma$ is curved,
the solution $u$ to~\eqref{problem:1} is however non-smooth across the interface. 
Here, we have to argument using smooth extensions 
of $u|_{\Omega_i}, i=1,2$ to the other sub-domain and the smallness 
of the region 
\begin{align*}
 S_h=(\Omega_1\cap \Omega_h^2) \cup (\Omega_2\cap \Omega_h^1)
\end{align*}
around the interface.

% \begin{figure}[t]
%   \centering
%   \input{curved}
%   \caption{Two different patches and two triangles, that are affected
%     by the interface intersection. The  modified finite element mesh
%     resolves the interface in a linear approximation.} 
%   \label{fig:curved}
% \end{figure}

The following result has been shown for the elliptic interface
problem~\eqref{problem:1}:

%{TW: Was ist interface problem 1?}

\begin{theorem}[A priori estimate]\label{thm:apriori}
  Let $\Omega\subset\mathbb{R}^2$ be a domain with convex polygonal
  boundary, split into $\Omega=\Omega_1\cup\Gamma\cup\Omega_2$, where
  $\Gamma$ is a smooth interface with $C^2$-parametrisation. We
  assume that $\Gamma$ divides $\Omega$ in such a way that the solution
  $u\in H^1_0(\Omega)$ satisfies the stability estimate
  \[
  u\in H^1_0(\Omega)\cap H^2(\Omega_1\cup\Omega_2),\quad
  \|u\|_{H^2(\Omega_1\cup\Omega_2)}\le c_s \|f\|. 
  \]
  For  the corresponding modified finite element solution $u_h\in V_h$, 
  it holds that
  \[
  \|\nabla (u-u_h)\|_\Omega\le C h_P \|f\|,\quad
  \|u-u_h\|_\Omega\le C h_P^2 \|f\|.
  \]
\end{theorem}
\begin{proof}
For the proof, we refer to \cite{RichterBuch} or \cite{FreiDiss}.
 \end{proof}

{
\section{Hierarchical basis functions}
\label{sec_hierarchical}

 The drawback of {the previously described} simple approach is that the condition number of the system matrix is unbounded
 for certain anisotropies ($r,s\to 0$). This is an unresolved issue in many of the presently used 
 enriched finite element methods for interface problems. We refer to \cite{LehrenfeldReusken} or \cite{BaBa12} 
 for 
 two of the few positive results
 in the case of extended finite elements of low-order. 
In our case, this can be circumvented by using a scaled hierarchical 
 finite element basis, that will yield system matrices
$A_h$ that satisfy the optimal bound
$\operatorname{cond}_2(A_h)=\mathcal{O}(h_P^{-2})$ for elliptic problems, with a constant that does not
depend on the position of the interface $\Gamma$ relative to the mesh elements.
  A detailed proof of this result has been given in~\cite{FreiRichter2014}.

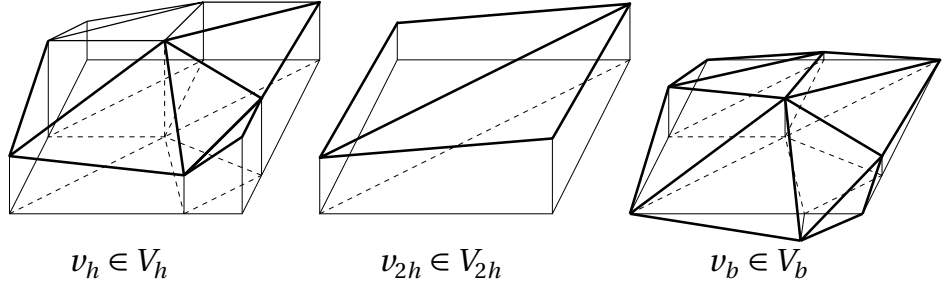
\begin{figure}[t]
  \centering
  \resizebox*{0.85\textwidth}{!}{
  
  \begin{picture}(0,0)%
\includegraphics{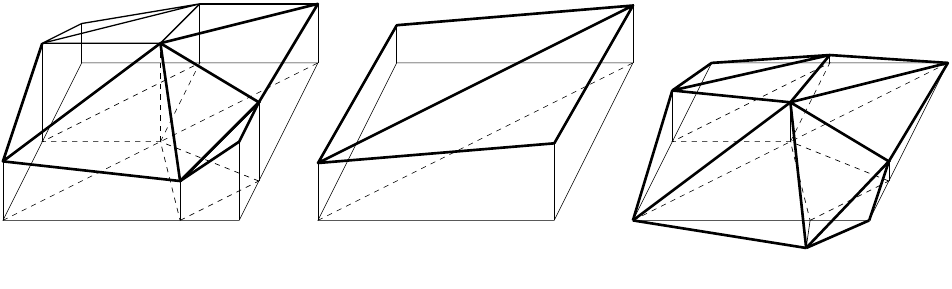}%
\end{picture}%
\setlength{\unitlength}{1657sp}%
\begingroup\makeatletter\ifx\SetFigFont\undefined%
\gdef\SetFigFont#1#2{%
  \fontsize{#1}{#2pt}%
  \selectfont}%
\fi\endgroup%
\begin{picture}(10866,3283)(2668,-5336)
\put(10801,-5236){\makebox(0,0)[lb]{\smash{{\SetFigFont{7}{8.4}{\color[rgb]{0,0,0}$v_b\in V_b$}%
}}}}
\put(6976,-5236){\makebox(0,0)[lb]{\smash{{\SetFigFont{7}{8.4}{\color[rgb]{0,0,0}$v_{2h}\in V_{2h}$}%
}}}}
\put(3376,-5236){\makebox(0,0)[lb]{\smash{{\SetFigFont{7}{8.4}{\color[rgb]{0,0,0}$v_h\in V_h$}%
}}}}
\end{picture}%
  
  }
  \caption{Example for a hierarchical splitting of a function $v_h\in
    V_h$ into coarse mesh part $v_{2h}\in V_{2h}$ and fine mesh fluctuation
    $v_b\in V_b$.}
  \label{fig:hierarchical}
\end{figure}

\begin{figure}[t]
  \centering
  %\begin{picture}(0,0)%
 \resizebox*{0.7\textwidth}{!}{
 
 \begin{picture}(0,0)%
\includegraphics{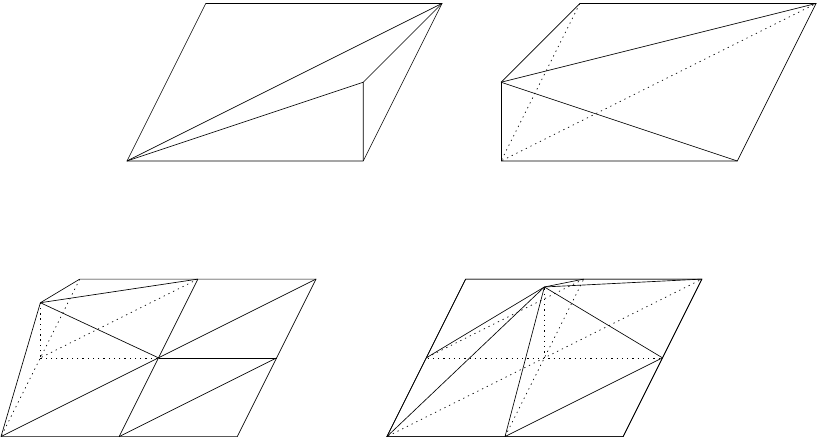}%
\end{picture}%
\setlength{\unitlength}{1657sp}%
\begingroup\makeatletter\ifx\SetFigFont\undefined%
\gdef\SetFigFont#1#2{%
  \fontsize{#1}{#2pt}%
  \selectfont}%
\fi\endgroup%
\begin{picture}(9339,4974)(6199,-4123)
\end{picture}%
 }%
% \end{picture}%
% \setlength{\unitlength}{1657sp}%
% %
% \begingroup\makeatletter\ifx\SetFigFont\undefined%
% \gdef\SetFigFont#1#2{%
%   \fontsize{#1}{#2pt}%
%   \selectfont}%
% \fi\endgroup%
% \begin{picture}(9339,4974)(6199,-4123)
% \end{picture}%

  \caption{Local basis functions of the hierarchical finite element space. 
  Top: Two of the four basis functions $\phi_i^{2h} \in V_{2h}$. Bottom: Two of the five basis functions
  $\phi_i^b \in V_b$}
  \label{fig:HierarchicalBasis}
\end{figure}

We split the finite element space $V_h$ in a hierarchical manner 
\[
V_h = V_{2h} + V_b,\quad  N:=\operatorname{dim}(V_h)=
\operatorname{dim}(V_{2h}) + \operatorname{dim}(V_b)=:N_{2h}+N_b.
\]
The space $V_{2h}$ is the standard space of piecewise bilinear or
linear functions on the patches $P\in\mathcal{T}_{2h}$ equipped with the usual
nodal Lagrange basis $V_{2h} =
\operatorname{span}\{\phi_{2h}^1,\dots,\phi_{2h}^{N_{2h}}\}$. Patches
cut by the interface are split into two large triangles. 

The space $V_b=V_h\setminus V_{2h}$ collects all functions, that are needed to
enrich $V_{2h}$ to $V_h$. These functions are defined piecewise on the sub-elements
in the remaining 5 degrees of freedom, see
Figure~\ref{fig:hierarchical} for an example of the splitting and Figure~\ref{fig:HierarchicalBasis}
for an illustration of the local basis functions. These basis functions are denoted by  
$V_b=\operatorname{span}\{\phi_b^1,\dots,\phi_b^{N_b}\}$. 
 The finite element space $V_{2h}$ on the other hand is fully isotropic
 and standard analysis holds. Functions
in $V_{2h}$ do not resolve the interface, while the basis functions
$\phi_b^i\in V_b$ will depend on the interface 
location if $\Gamma\subset \operatorname{supp}\,\phi_b^i$. 

In order to define the hierarchical ansatz
space, we have to modify some of the basic triangles in the cases A, B and C,
see Figure~\ref{basisfunctions}. In contrast to Section~\ref{sec:fe}, the midpoint can be moved along one 
of the diagonal lines only, such that the space $V_{2h}$ can be defined as
space of piecewise linear functions on two large triangles. 
Note that in order to guarantee a maximum angle condition
in the cases A.1 and C.1 in Figure~\ref{basisfunctions}, we must also move 
the outer node $x_2$ belonging to the space $V_b$, due to the additional constraint on the position of $m_P$.

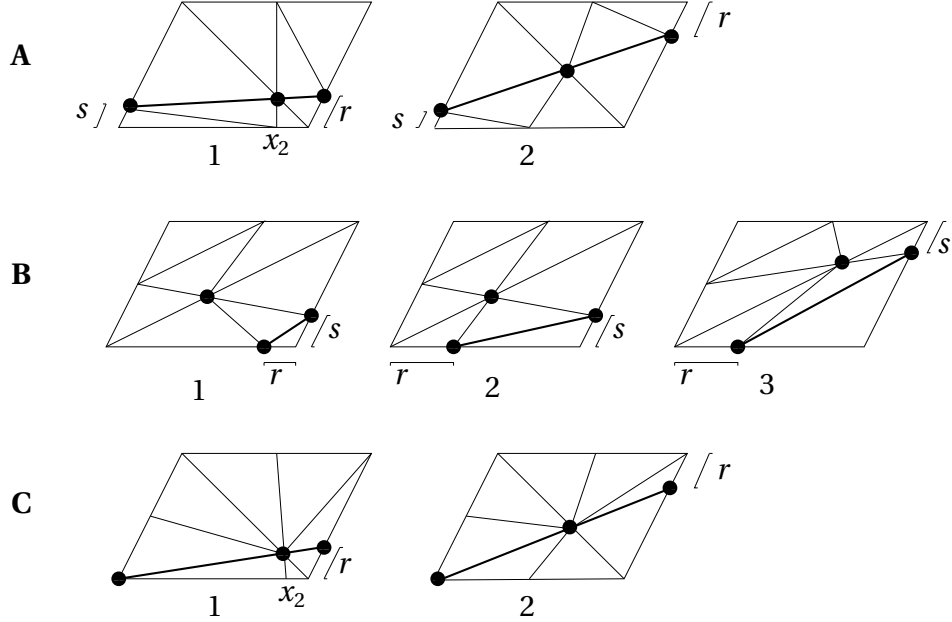
\begin{figure}[t]
  \centering
  \resizebox*{0.85\textwidth}{!}{
  
  \begin{picture}(0,0)%
\includegraphics{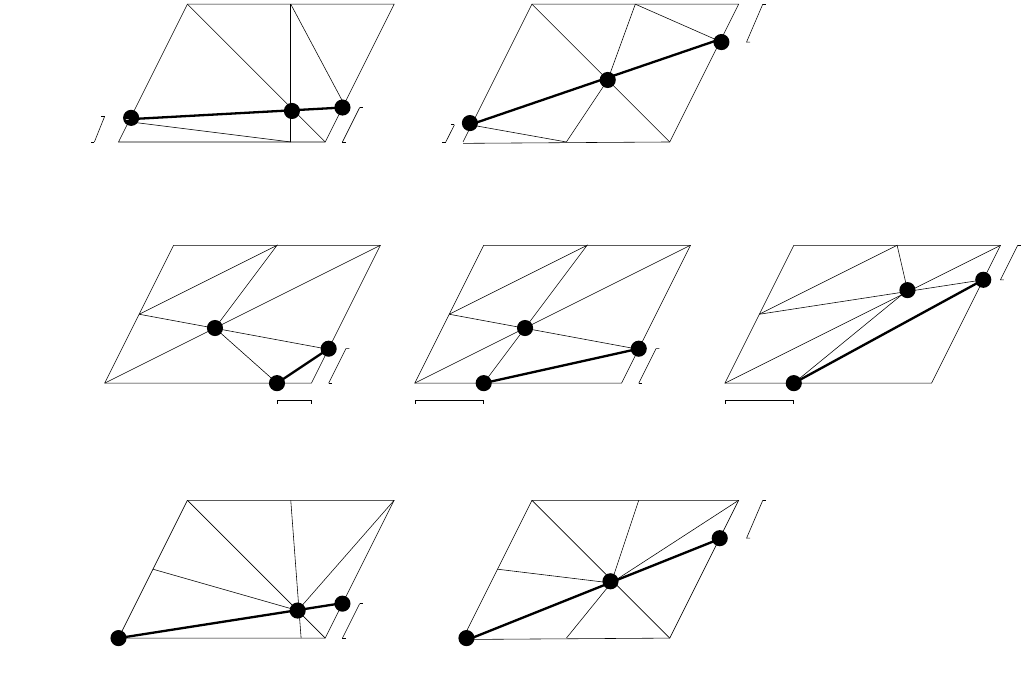}%
\end{picture}%
\setlength{\unitlength}{1450sp}%
\begingroup\makeatletter\ifx\SetFigFont\undefined%
\gdef\SetFigFont#1#2{%
  \fontsize{#1}{#2pt}%
  \selectfont}%
\fi\endgroup%
\begin{picture}(13347,8847)(5836,-7996)
\put(9496,-4651){\makebox(0,0)[lb]{\smash{{\SetFigFont{6}{7.2}{\color[rgb]{0,0,0}$r$}%
}}}}
\put(10351,-3976){\makebox(0,0)[lb]{\smash{{\SetFigFont{6}{7.2}{\color[rgb]{0,0,0}$s$}%
}}}}
\put(15346,-4651){\makebox(0,0)[lb]{\smash{{\SetFigFont{6}{7.2}{\color[rgb]{0,0,0}$r$}%
}}}}
\put(19036,-2761){\makebox(0,0)[lb]{\smash{{\SetFigFont{6}{7.2}{\color[rgb]{0,0,0}$s$}%
}}}}
\put(14401,-3976){\makebox(0,0)[lb]{\smash{{\SetFigFont{6}{7.2}{\color[rgb]{0,0,0}$s$}%
}}}}
\put(11296,-4651){\makebox(0,0)[lb]{\smash{{\SetFigFont{6}{7.2}{\color[rgb]{0,0,0}$r$}%
}}}}
\put(5851,-3211){\makebox(0,0)[lb]{\smash{{\SetFigFont{8}{9.6}{\color[rgb]{0,0,0}\textbf{B}}%
}}}}
\put(5851,-61){\makebox(0,0)[lb]{\smash{{\SetFigFont{8}{9.6}{\color[rgb]{0,0,0}\textbf{A}}%
}}}}
\put(16516,-4786){\makebox(0,0)[lb]{\smash{{\SetFigFont{8}{9.6}{\color[rgb]{0,0,0}3}%
}}}}
\put(12601,-4831){\makebox(0,0)[lb]{\smash{{\SetFigFont{8}{9.6}{\color[rgb]{0,0,0}2}%
}}}}
\put(8416,-4876){\makebox(0,0)[lb]{\smash{{\SetFigFont{8}{9.6}{\color[rgb]{0,0,0}1}%
}}}}
\put(5851,-6496){\makebox(0,0)[lb]{\smash{{\SetFigFont{8}{9.6}{\color[rgb]{0,0,0}\textbf{C}}%
}}}}
\put(15886,-6001){\makebox(0,0)[lb]{\smash{{\SetFigFont{6}{7.2}{\color[rgb]{0,0,0}$r$}%
}}}}
\put(10486,-7351){\makebox(0,0)[lb]{\smash{{\SetFigFont{6}{7.2}{\color[rgb]{0,0,0}$r$}%
}}}}
\put(8641,-7981){\makebox(0,0)[lb]{\smash{{\SetFigFont{8}{9.6}{\color[rgb]{0,0,0}1}%
}}}}
\put(13096,-7981){\makebox(0,0)[lb]{\smash{{\SetFigFont{8}{9.6}{\color[rgb]{0,0,0}2}%
}}}}
\put(9631,-7756){\makebox(0,0)[lb]{\smash{{\SetFigFont{5}{6.0}{\color[rgb]{0,0,0}$x_2$}%
}}}}
\put(6751,-826){\makebox(0,0)[lb]{\smash{{\SetFigFont{6}{7.2}{\color[rgb]{0,0,0}$s$}%
}}}}
\put(11251,-961){\makebox(0,0)[lb]{\smash{{\SetFigFont{6}{7.2}{\color[rgb]{0,0,0}$s$}%
}}}}
\put(15886,479){\makebox(0,0)[lb]{\smash{{\SetFigFont{6}{7.2}{\color[rgb]{0,0,0}$r$}%
}}}}
\put(10486,-871){\makebox(0,0)[lb]{\smash{{\SetFigFont{6}{7.2}{\color[rgb]{0,0,0}$r$}%
}}}}
\put(8641,-1501){\makebox(0,0)[lb]{\smash{{\SetFigFont{8}{9.6}{\color[rgb]{0,0,0}1}%
}}}}
\put(13096,-1501){\makebox(0,0)[lb]{\smash{{\SetFigFont{8}{9.6}{\color[rgb]{0,0,0}2}%
}}}}
\put(9406,-1276){\makebox(0,0)[lb]{\smash{{\SetFigFont{5}{6.0}{\color[rgb]{0,0,0}$x_2$}%
}}}}
\end{picture}%

  }
  
  \caption{Configuration of the hierarchical basis functions $V_b$ for
    the different patch types. In each sketch, we consider the case
    $r\to 0$ or $s\to 0$ or both.}
  \label{basisfunctions}
\end{figure}

}

\subsection*{Scaling of the basis functions}

Moreover, in order to ensure the optimal bound for the condition number, we have to normalise
the Lagrangian basis functions on the fine scale $\phi_b^i, i=1,...,N_b,$ 
by setting
\begin{align*}
 \tilde{\phi}_b^i := \frac{\phi_b^i}{\|\nabla \phi_b^i\|},
\end{align*}
such that it holds that
  \begin{equation}\label{scaling}
    C^{-1}\le  \|\nabla\tilde{\phi}_b^i\| \le C,\quad
    i=1,\dots,N_b.
  \end{equation}
  
  In a practical implementation,
  one can use the basis
  $\phi_i, i=1,\ldots,N$ to assemble the system matrix $A_h$ and apply
  a simple row- and column-wise scaling with the diagonal elements
  \[
  a_{ij} = (\nabla\phi_j,\nabla\phi_i),\quad
  \tilde a_{ij}:=\frac{a_{ij}}{\sqrt{a_{ii} a_{jj}}}.
  \]
  Alternatively, a simple preconditioning of the 
  linear system can be applied multiplying with the diagonal of the system matrix from left and right
  \[
  \mathbf{A}\mathbf{x} = \mathbf{b} \quad\Leftrightarrow\quad
  \mathbf{D}^{-\frac{1}{2}}\mathbf{A}\mathbf{D}^{-\frac{1}{2}} 
  \widetilde{\mathbf{x}} = D^{-\frac{1}{2}} b,\quad \widetilde{\mathbf{x}}
  = \mathbf{D}^\frac{1}{2}\mathbf{x},
  \]
  where $\mathbf{D} = \operatorname{diag}(a_{ii})$.

%%%%%%%%%%%%%%%%%%%%%%%%%%%%%%%%%%%%%%%%%%%%%%%%%%%%%%%%%%%%%%%%%%%%%%%%%
\section{Numerical examples}
\label{sec:num}

We now present two numerical examples that
include all different types of interface
cuts (configurations A to D) and arbitrary anisotropies.

\subsection{Example 1: Performance under mesh refinement}
\label{sec_ex_1}
This first example has already been considered to discuss the
interface approximation in Section~\ref{sec:fe}, see Figure~\ref{fig:standardfe} for a
sketch of the configuration. The unit square $\Omega=(-1,1)^2$ is 
split into a ball $\Omega_1=B_R(x_m)$ with radius $R=0.5$, midpoint
$x_m=(0,0)$ and
$\Omega_2=\Omega\setminus\bar\Omega_1$. As diffusion parameters we
choose $\kappa_1=0.1$ and $\kappa_2=1$. We
use the analytical solution
\[
u(x)=\begin{cases}
    -2 \kappa_2 \|x-x_m\|^4, \quad &x \in \Omega_2,\\
    -\kappa_1 \|x-x_m\|^2 +\frac{1}{4}\kappa_1 - \frac{1}{8}\kappa_2
  \quad &x \in \Omega_1,
\end{cases}
\]
to define the right-hand side $f_i:=-\kappa_i\Delta u$ in $\Omega_i$ and the Dirichlet
boundary data. A sketch of
the solution is given on the right side of Figure~\ref{fig:ex1_a}. 

\begin{figure}[t]
\centering
{\includegraphics[width=7cm]{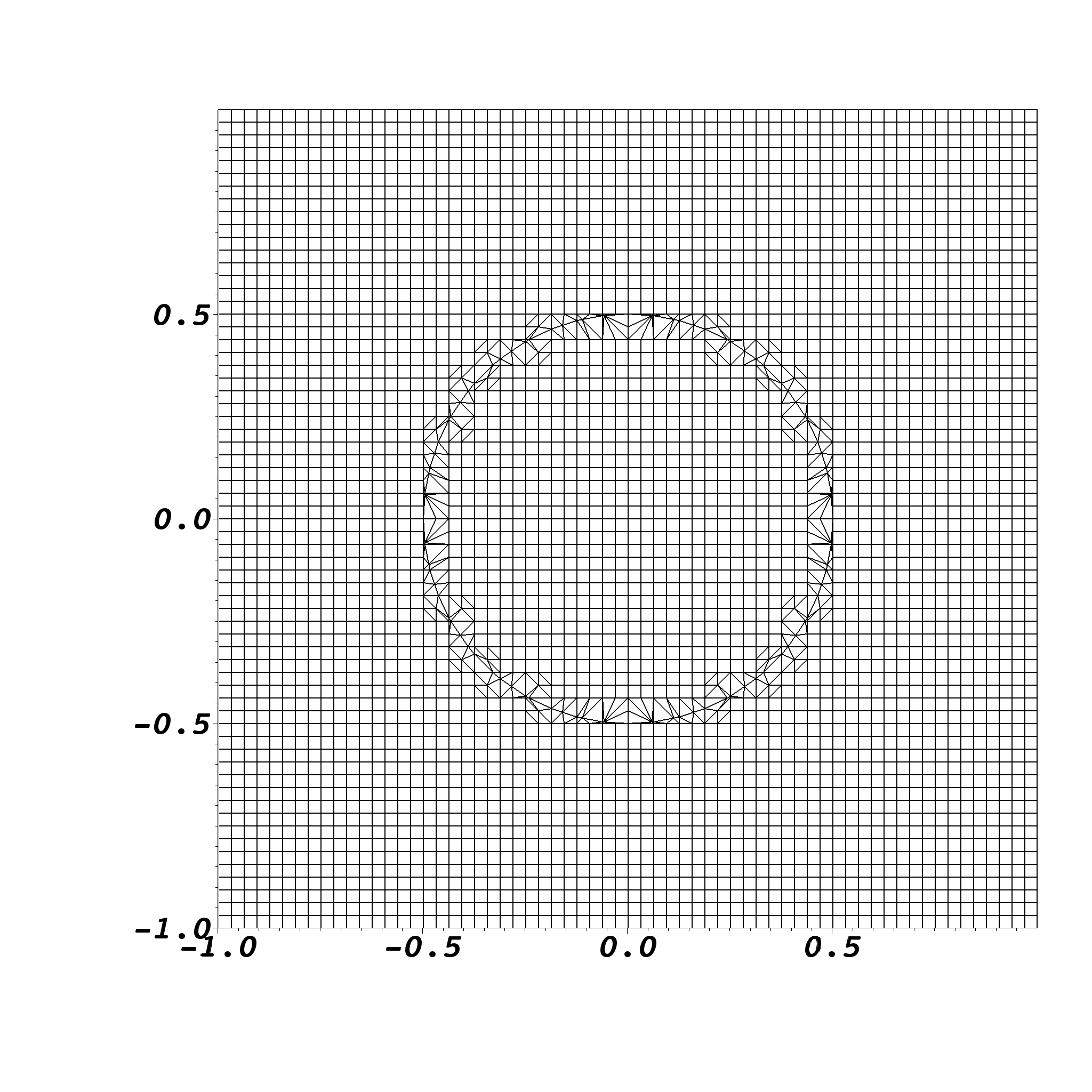}}
{\includegraphics[width=7cm]{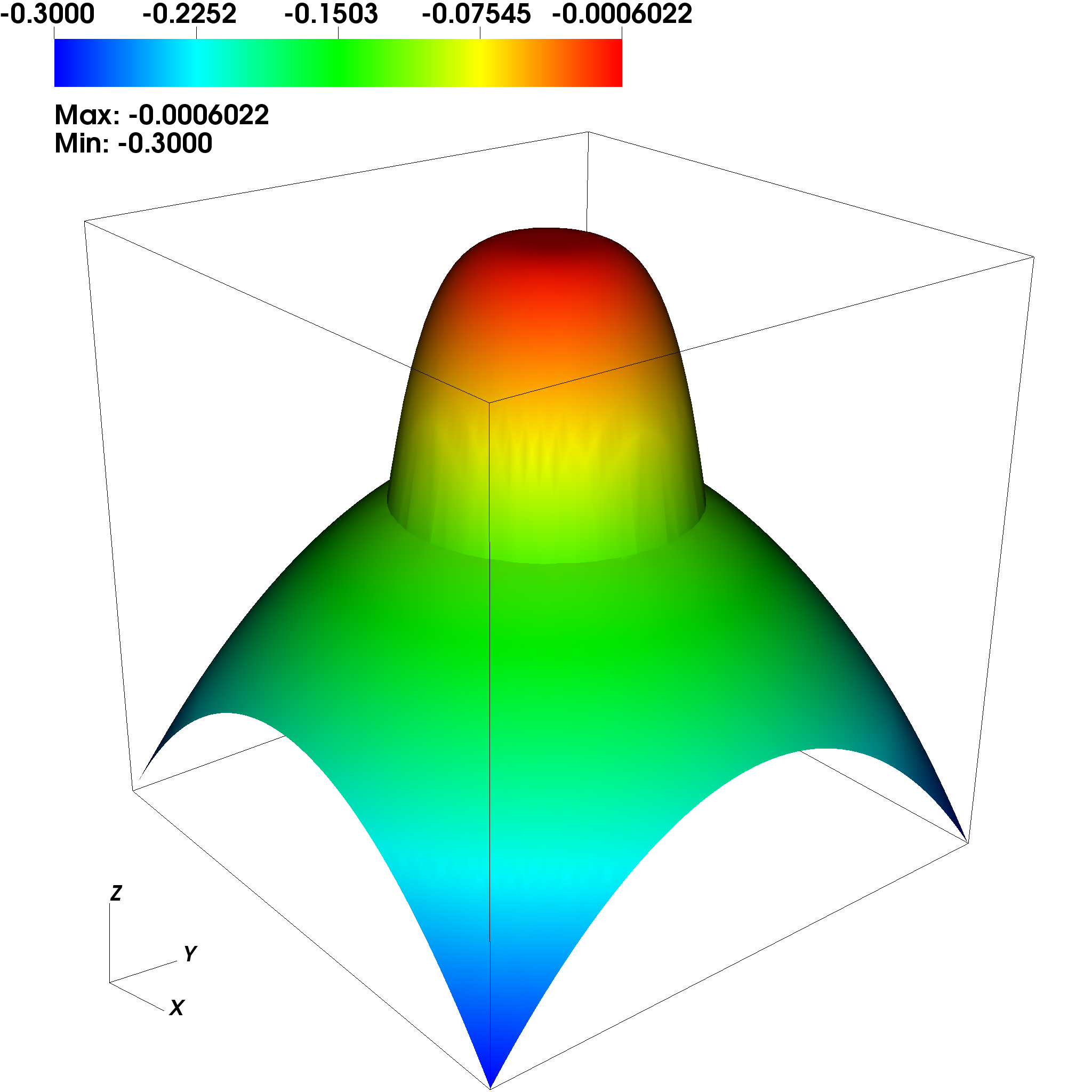}}
\caption{{Example 1: The cut-mesh on level $4$ (left) and a 3D surface plot of the solution (right).}}
\label{fig:ex1_a}
\end{figure}

On the coarsest mesh with 16 patch elements, we have four patches of type D.
After some steps of global refinement this simple example includes the
configurations A to C with different anisotropies. 
In Figure~\ref{fig:ball}, we plot the
$H^1$- and $L^2$-norm errors obtained on several levels of global mesh
refinement. According to Theorem~\ref{thm:apriori}, we observe linear convergence in the
$H^1$-norm and quadratic convergence in the $L^2$-norm.  
For comparison, Figure~\ref{fig:standardfe} shows the corresponding
results using standard non-fitting finite elements. 

\begin{figure}[t]
  \centering
  \begin{minipage}{0.65\textwidth}
    \includegraphics[width=\textwidth]{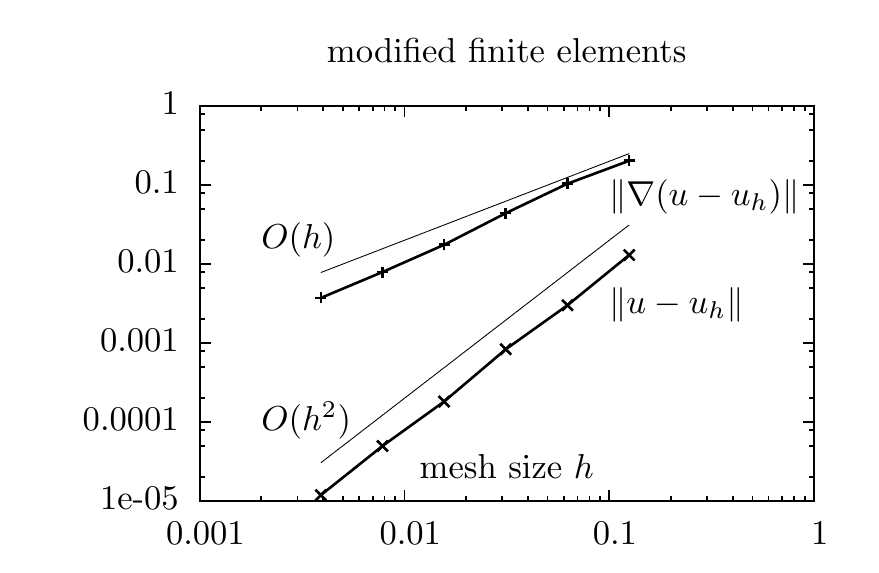}
  \end{minipage}
%  \hspace{0.025\textwidth}
%  \begin{minipage}{0.4\textwidth}
%    \vspace{0.8cm}
%    \includegraphics[width=\textwidth]{circle_color_modfe.png}
%  \end{minipage}
  
  \caption{Example 1: $H^1$- and $L^2$ errors under mesh
    refinement. %\textit{Right:} Sketch of the solution.
}\label{fig:ball}    
\end{figure} 

{ As we have shown numerically computed condition numbers for this and the following
example already in \cite{FreiRichter2014}, we provide here
computational evidence that the arising linear systems can be solved with iterative methods 
as the conjugate gradient (CG) instead. We incorporate the scaling of the
basis functions by means of a diagonal preconditioner, as discussed 
in Section~\ref{sec_hierarchical}. In order to analyse the effect of 
the scaling, we compare the performance of the diagonally preconditioned CG method
(dPCG) with a standard CG scheme without preconditioning. Moreover, we also show 
the performance of a CG scheme with SSOR relaxation as preconditioner (SSOR-PCG, 
without a scaling of the basis functions).
For the latter we choose the relaxation parameter 
$\omega = 1.2$, see e.g., \cite{Meister2011}. 
The (absolute) tolerance for the global residual is chosen as $10^{-12}$.

The iteration numbers 
for the non-hierarchical finite element basis introduced in Section~\ref{sec:fe} (nh) and the hierarchical (h) 
variant
described in Section~\ref{sec_hierarchical} in combination 
with the three CG methods are shown in Table~\ref{tab_CG_iter} on different mesh levels,
where each finer mesh is constructed from the coarser one by global mesh refinement.

Theoretically the number of iterations needed to reach a certain tolerance in the CG method
should scale with the square root of the condition number $\mathcal{O}(\sqrt{\kappa})$
(see e.g., \cite{Br07,BuchRichterWick}), i.e.$\,$for
the scaled hierarchical approach with a condition number of order $\kappa = \mathcal{O}(h_P^{-2})$, 
we can expect that the number of iterations grows asymptotically with $\mathcal{O}(h_P^{-1})$. This behaviour
can be observed quite clearly for the preconditioned CG methods in Table~\ref{tab_CG_iter}. The SSOR 
preconditioning seems to work even better than the diagonal preconditioning.
In this example, the expected convergence of the linear solver can be obtained without using 
the hierarchical basis functions.
The use of the hierarchical basis leads however to an advantage
in terms of the absolute numbers of iterations.

For the standard CG method without preconditioning, we observe that the number of iterations
grows faster than $\mathcal{O}(h_P^{-1})$ for both the hierarchical and the non-hierarchical approach.
This has to be expected, as 
the condition number might be unbounded for certain anisotropies. The observation that the iteration numbers
for the scaled non-hierarchical approach seem bounded by $\mathcal{O}(h_P^{-1})$ in this example,
might be due to the fact that not all kind of anisotropies are present
and that the anisotropies that are present do not
necessarily get worse on the finer grids. To study the performance of our approach considering all
kinds of anisotropies (see Figure~\ref{basisfunctions}), we will next move the circular interface 
gradually by small fractions of patch cells in vertical direction.
}

%{ TO DO in Code: Choice of preconditioning in parameter file}
% twick: done!

 \begin{table}
   \begin{center}
     \begin{tabular}{cc|ccc|ccc}
       \toprule
        Level & \#Patches  & CG(nh) & dPCG(nh) &SSOR-PCG(nh) & CG(h) & dPCG(h) &SSOR-PCG(h)  \\ \hline
       $0$   & 16    & 10  &10  &15  & 10 &10  &15\\
       $1$   & 64    & 43   &29 &32  &64  &39  &25\\
       $2$   & 256   & 114 &60  &56  &126 &61  &32 \\
       $3$   & 1024  & 253 &124 &97  &197 &95  &47\\
       $4$   & 4096  & 561 &238 &175 &351 &167 &81 \\
       $5$   & 16384 &1436 &484 &335 &881 &322 &150\\
       $6$   & 65536 &3518 &967 &634 &2053 &622 &293\\
       \bottomrule
     \end{tabular}
   \end{center}
     \caption{Example 1: Iteration numbers of the linear solvers on different mesh levels for hierarchical (h) and
       non-hierarchical (nh) versions and the standard CG method compared to a diagonally preconditioned (dPCG)
       and a SSOR-preconditioned CG (SSOR-PCG) approach.}
     \label{tab_CG_iter}
 \end{table}

\subsection{Example 2: Performance for different anisotropies}
\label{sec_ex_2}
To include all kind of anisotropies, we fix the refinement level to the fourth level
of the previous example (4096 patch cells) and move the circular interface gradually in vertical direction. 
Precisely, we move the position of the midpoint by
\begin{align*}
 x_m= (0, \frac{k}{N} h_P)
\end{align*}
for $k=0,...,N-1$, where $N=1000$. Note that for $k=N$, the interface would have been moved by exactly one patch cell,
i.e.$\,$exactly the same cuts as for $k=0$ would appear. The problem and parameters are exactly the same as in the previous
example (note that the exact solution and the data defined above depend on $x_m$). 

The meshes 
for $k=0$ and $k=990$ are shown in Figure \ref{fig:ex2_b}.
Moreover, in order to illustrate the anisotropic sub-cells, a zoom-in of the cut-meshes for $k=0,10,50$ and $990$ is 
displayed in larger in Figure~\ref{fig:ex2_c}.
For $k=0$, we find very anisotropic cells in two patches of 
type C in the patches in the centre; for $k=10$ in four patches of type B; for $k=50$ in two patches of type B 
in the middle and two patches of type A on the left and right; for $k=990$ very anisotropic cells of 
type A are present.

\begin{figure}[t]
\centering
{\includegraphics[width=7cm]{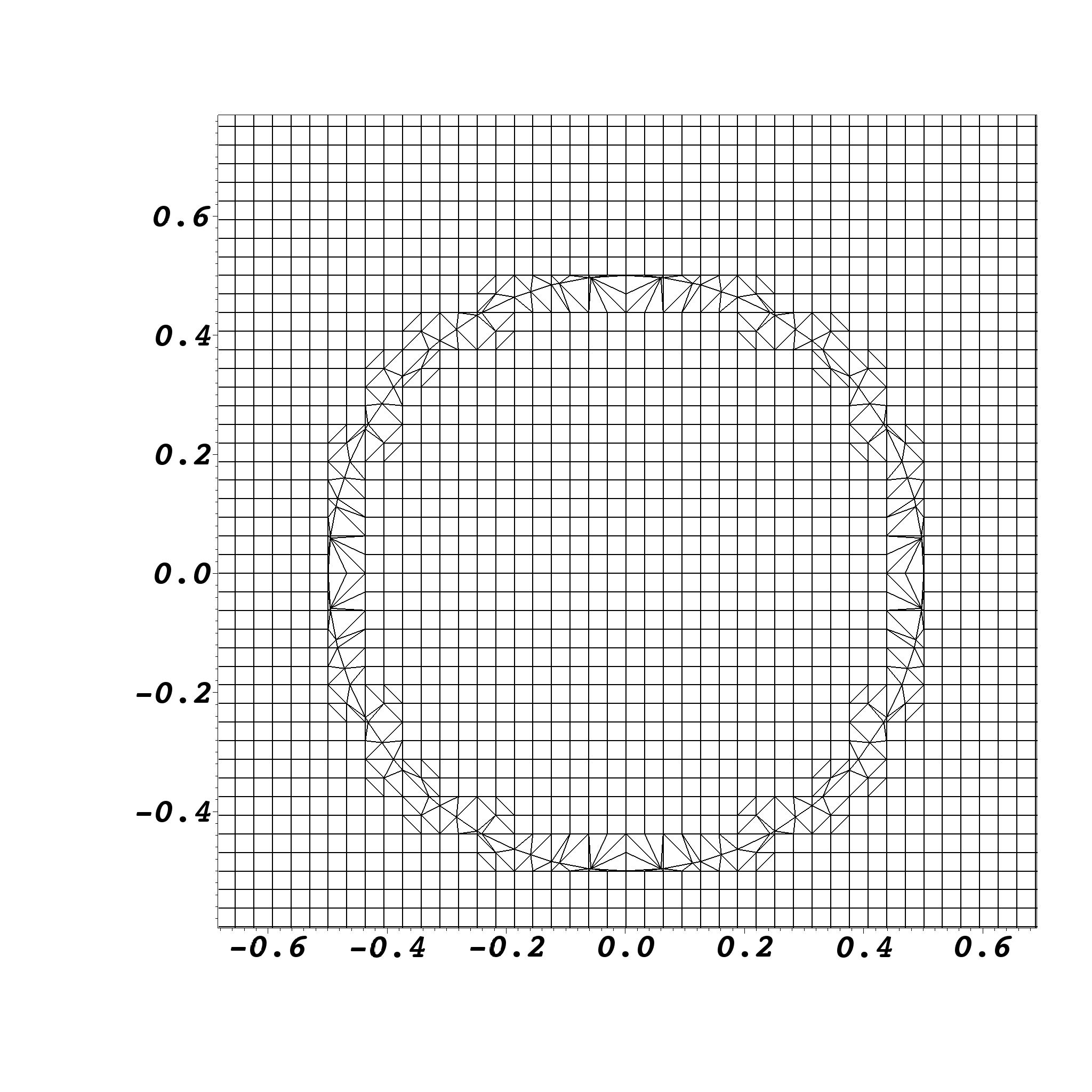}}
{\includegraphics[width=7cm]{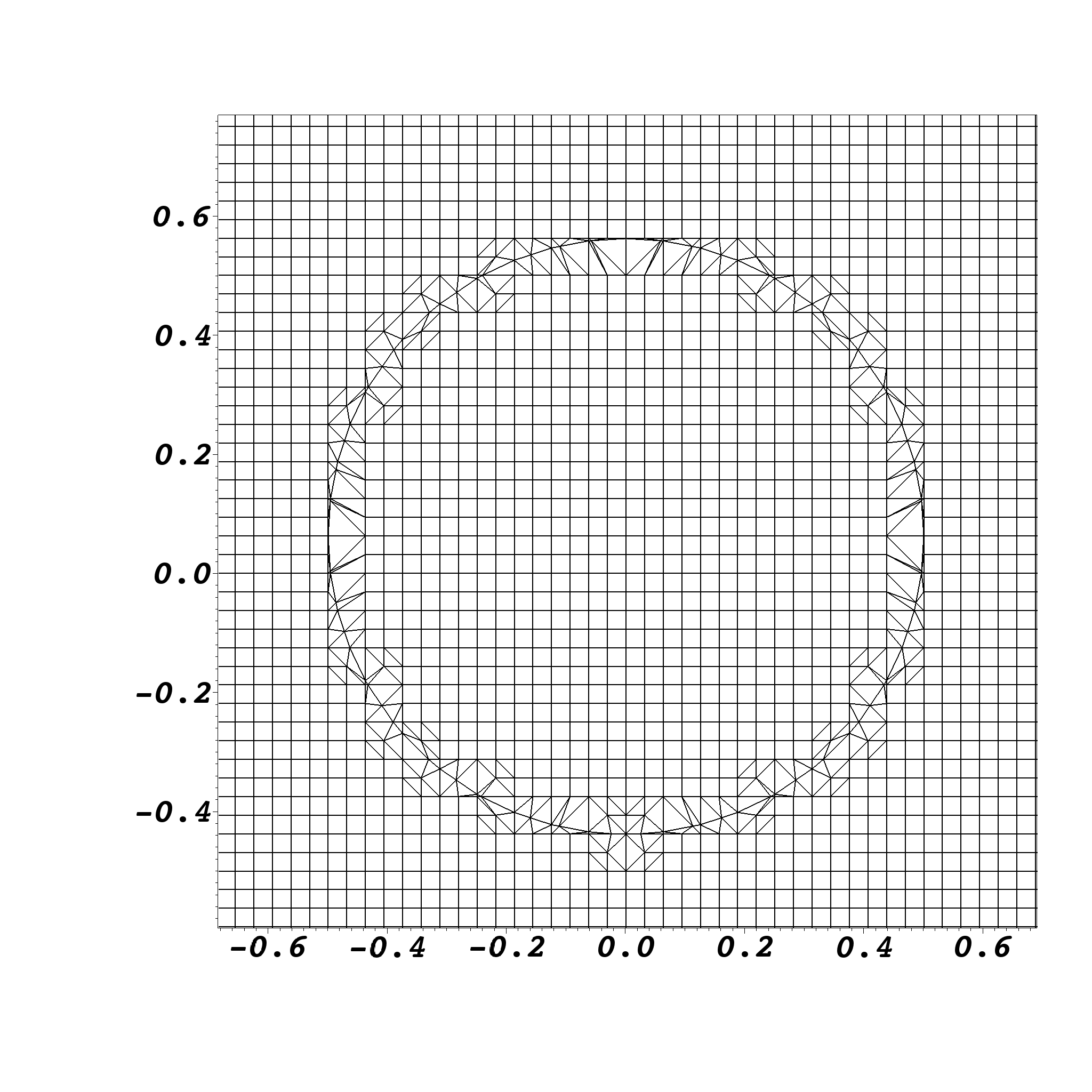}}
\caption{Example 2: The cut-mesh at $k=0$ and $k=990$.}
\label{fig:ex2_b}
\end{figure}

\begin{figure}[t]
\centering
{\includegraphics[width=6.5cm]{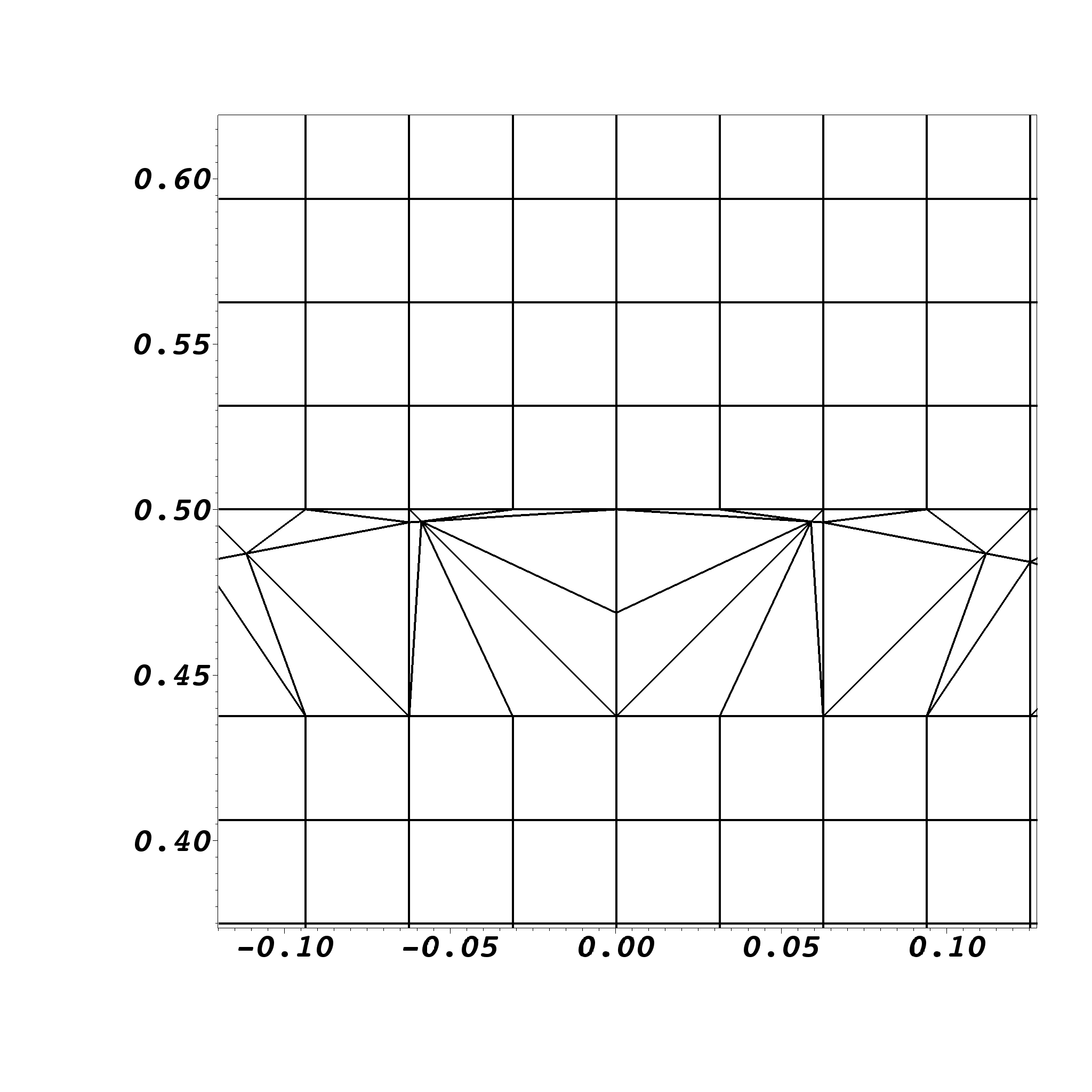}}
{\includegraphics[width=6.5cm]{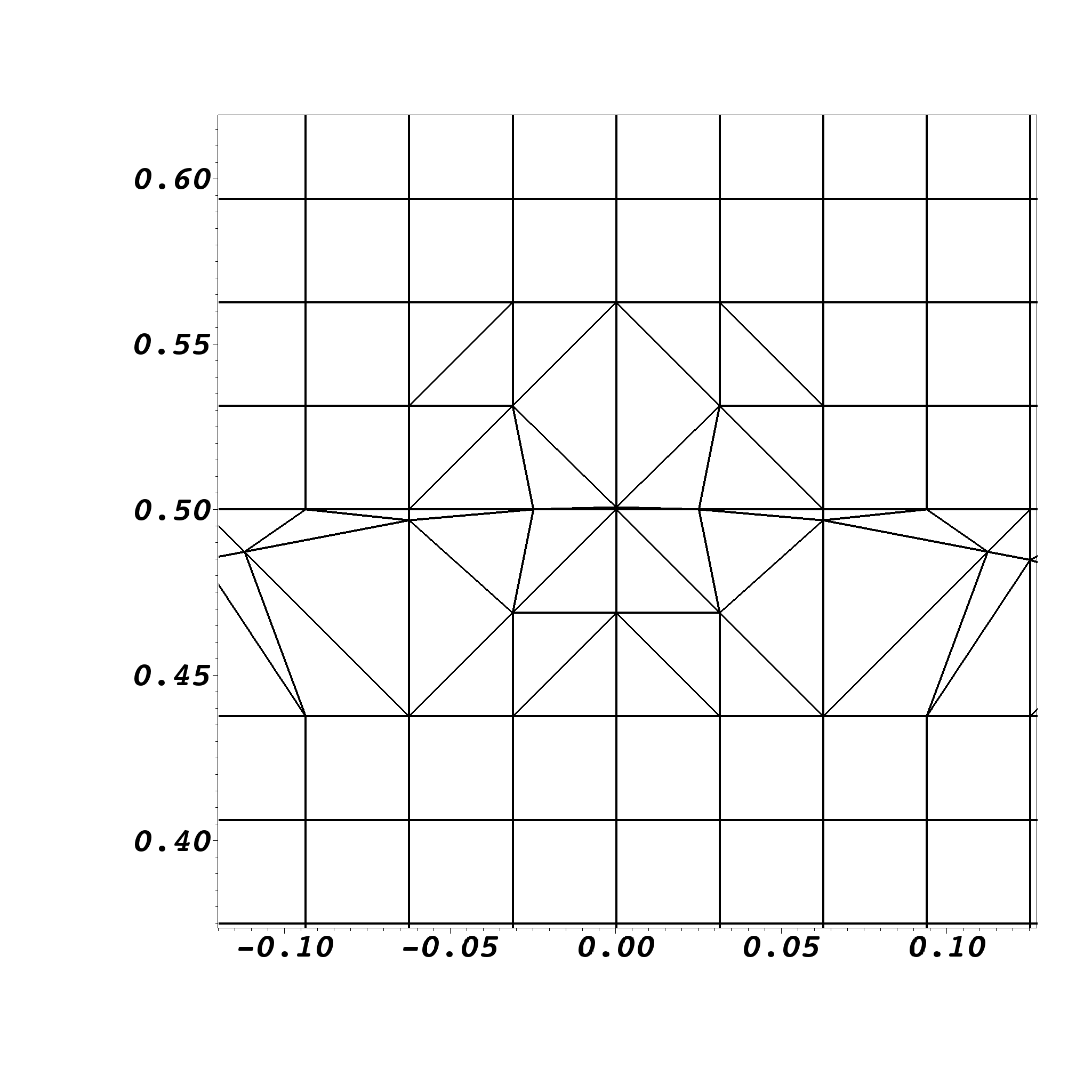}}
{\includegraphics[width=6.5cm]{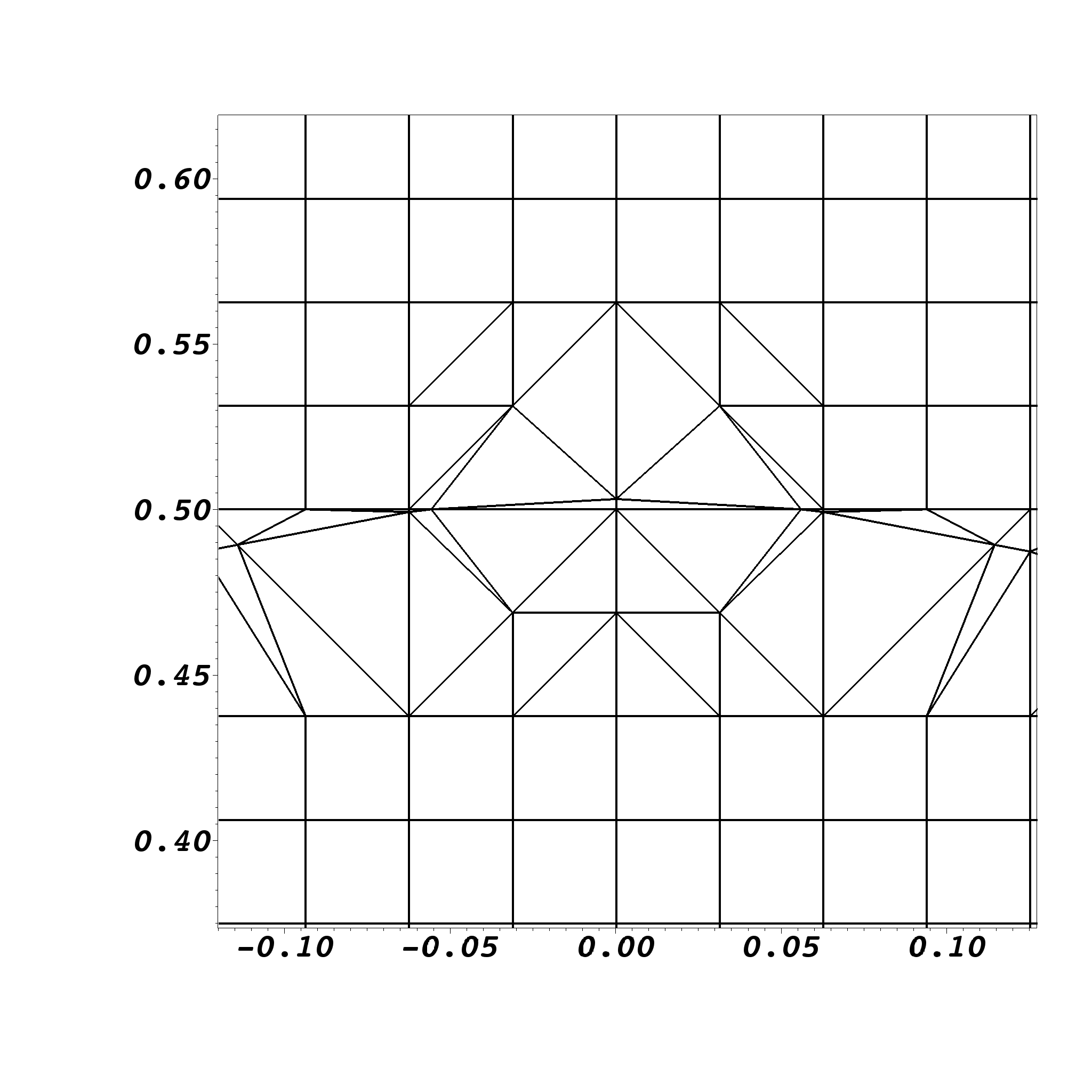}}
{\includegraphics[width=6.5cm]{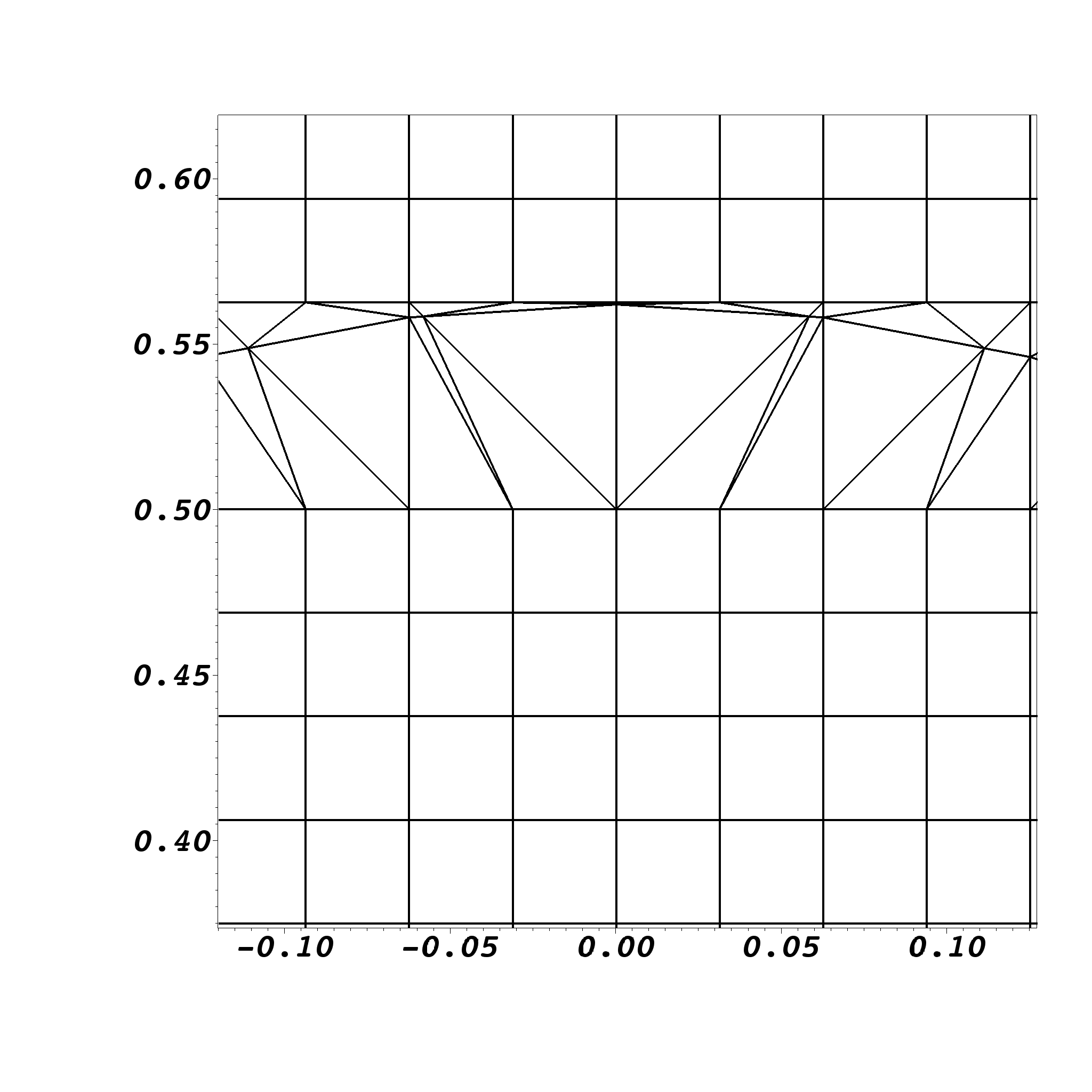}}
\caption{Example 2: Zoom-in at $k=0$ (top left), $k=10$ (top right), 
$k=50$ (bottom left) and $k=990$ (bottom right).}
\label{fig:ex2_c}
\end{figure}

{In Table~\ref{tab.aniso}, we show some properties of the triangulation $\mathcal{T}_h$ consisting of the sub-cells for the four different configurations
shown in Figure~\ref{fig:ex2_c}. The most
anisotropic cells can be found for $k=10$ and $k=990$, where both the largest aspect ratio
\begin{align*}
 \max\limits_{K\in\mathcal{T}_h} \frac{|e_{K,\text{max}}|}{|e_{K,\text{min}}|}
\end{align*}
of an element and the ratio between the largest and the smallest element's size are of order $10^5$. 
Note that due to the symmetry of the problem and the discretisation, the values for $k=10$ and $k=990$ are identical.
The element with the largest aspect ratio can be found on the very left of the circle (and due to symmetry also on the very right, 
see Figure~\ref{fig:ex2_b} on the right), 
where the patch line connecting the vertices $x_1=(-0.5, 0.03125)$ and $x_2=(-0.46875, 0.03125)$ is cut by the interface at $x_s\approx(-0.4999999,0.03125)$.}

\begin{table}
\begin{center}
     \begin{tabular}{c|ccc|ccc}
       \toprule
        k &$|K_{\max}|$ &$|K_{\min}|$ & $\frac{|K_{\max}|}{|K_{\min}|}$ &  $|e_{\text{max}}|$ &  $|e_{\text{min}}|$ & $\max\limits_{K\in\mathcal{T}_h} \frac{|e_{K,\text{max}}|}{|e_{K,\text{min}}|}$ \\ \hline
        0 &$2.44\cdot 10^{-4}$ &$3.82\cdot 10^{-\,6\,\,}$ &$6.39\cdot 10^1$ &$3.45\cdot 10^{-2}$ &$4.89\cdot 10^{-4}$ &$3.20\cdot 10^1$ \\
        10 &$2.50\cdot 10^{-4}$ &$7.63\cdot 10^{-10}$ &$3.28\cdot10^5$ &$3.45\cdot 10^{-2}$ &$9.77\cdot 10^{-8}$ &$1.60\cdot 10^5$\\
        50 &$2.52\cdot 10^{-4}$ &$1.91\cdot 10^{-\,8\,\,}$ &$1.32\cdot10^4$ & $3.45\cdot10^{-2}$ &$2.44\cdot10^{-6}$&$6.40\cdot10^3$ \\
       990 &$2.50\cdot 10^{-4}$ &$7.63\cdot 10^{-10}$ &$3.28\cdot10^5$ &$3.45\cdot 10^{-2}$ &$9.77\cdot 10^{-8}$ &$1.60\cdot 10^5$\\
       \bottomrule
     \end{tabular}
   \end{center}
     \caption{\label{tab.aniso} Properties of the triangulations $\mathcal{T}_h$ consisting of the sub-cells for the four different configurations shown 
     in Figure~\ref{fig:ex2_c}. In columns 2 to 4, we show the area of the largest and the
     smallest element $|K_{\max}|$ and $|K_{\min}|$ and their ratio; in columns 5 and 6 he largest and smallest edge $|e_{\text{max}}|$ and $|e_{\text{min}}|$. 
     Finally, in column 7 the biggest aspect ratio of all elements is shown.     
     }
 
\end{table}

In order to study the dependence of the iteration numbers on the position of the interface, we plot
the number of linear iterations for the three different CG methods and the non-hierarchical and hierarchical basis
in Figure \ref{fig:ex2_a} over the increment $k$. For both the non-hierarchical and the hierarchical approach, we
observe that the iteration numbers decrease by at least a factor of 2 for the diagonal preconditioning and at least by a 
factor of 4 for the SSOR preconditioning compared to the standard CG method.

For the non-hierarchical approach, the iteration numbers depend considerably on the position of the interface,
even after preconditioning. Using the diagonal preconditioning the iteration number varies
between 239 and 585 iterations, for the SSOR preconditioning between 129 and 260 iterations are needed. These numbers 
get worse, when the fineness $N$ is increased. In this example it becomes clear that the non-hierarchical approach shows
a condition number issue, even when preconditioning techniques are used.

For the hierarchical approach the iteration numbers seem to be bounded independently of the position of the interface for
both preconditioning variants. The diagonally preconditioned CG method needs between 163 and 188 linear iterations, the
SSOR preconditioned 
CG method between 62 and 81 iterations. Again the SSOR preconditioned CG method is superior to the simple diagonal 
preconditioning, although our analysis for the condition number is based on the scaling of the hierarchical basis
\eqref{scaling}, which is
only ensured for the diagonal preconditioning.

\begin{figure}[t]
\centering
 \begin{minipage}{0.5\textwidth}
  \includegraphics[width=\textwidth]{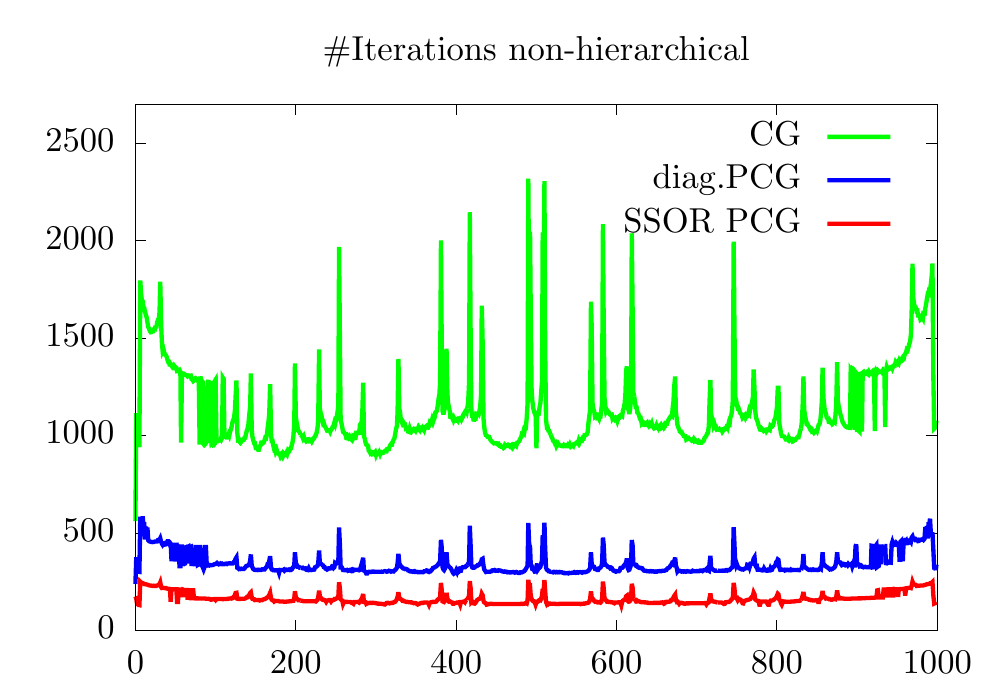}
 \end{minipage}
 \hspace{-0.5cm}
  \begin{minipage}{0.5\textwidth}
    \includegraphics[width=\textwidth]{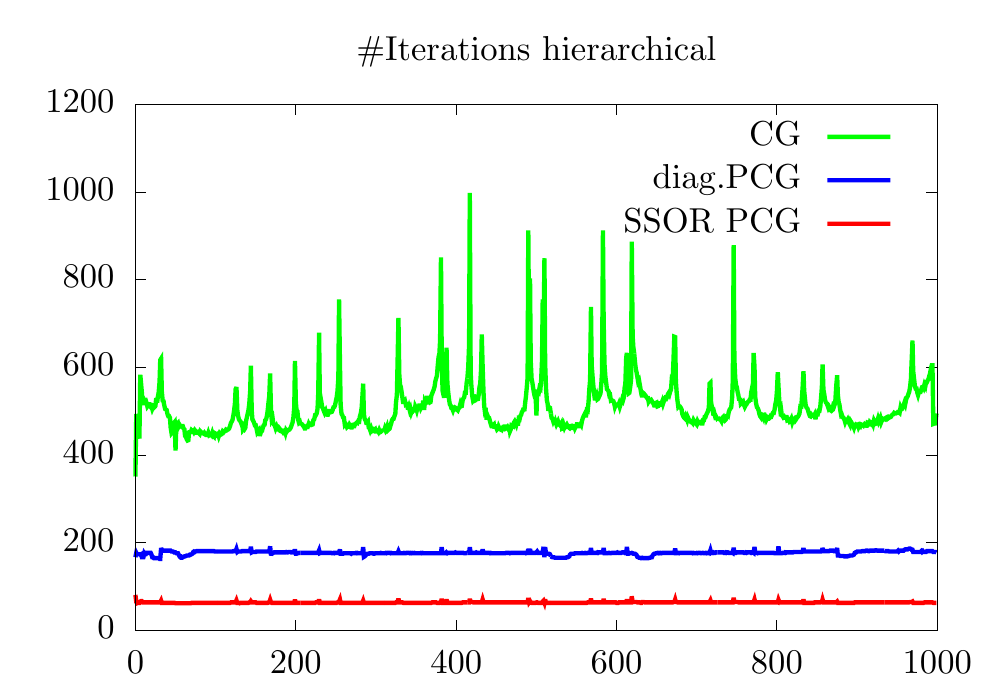}
 \end{minipage}
%{\includegraphics[width=10cm]{cg_iter_case_2_plot.pdf}}
\caption{Example 2: Number of linear iterations needed for the different CG methods to decrease the residual
below a tolerance of $10^{-12}$ plotted
over the increment $k$, where for $k=1000$ the circular interface has been moved by exactly one patch cell.
\textit{Left}: Non-hierarchical finite element basis. \textit{Right}: Hierarchical basis. }
\label{fig:ex2_a}
\end{figure}

%%%%%%%%%%%%%%%%%%%%%%%%%%%%%%%%%%%%%%%%%%%%%%%%%%%%%%%%%%%%%%%%%%%%%%%%%%%%%%%%%%
\section{Implementation}
\label{sec.impl}

% \subsection{Files}
% The code comes with the following files:
% \begin{lstlisting}
% step-modfe.cc              // Source file similar to many deal.II tutorials
% locmodfe.h                 // Described in more detail in the Appendices B and C
% locmodfe.cc                // Described in more detail in the Appendices B and C
% problem.h                  // Problem-specific definition of data and solution
% unit_square.inp            // Grid file
% parameters_test_case_1.prm // Various grid, model and material parameters for case 1
% parameters_test_case_2.prm // Various grid, model and material parameters for case 2
% test_case_1.dlog           // Terminal output of case 1
% test_case_2.dlog           // Terminal output of case 2
% CMakeLists.txt             // To build the Makefile using cmake
% copyright.txt              // Copyright and license information
% \end{lstlisting}

%\subsection{Installation and compilation}
Our implementation is based on deal.II, version 8.5.0. A short guide on the installation and compilation 
is given in the file \texttt{README.txt}.

We start this section by giving an overview of the basic structure of the source code in Section~\ref{sec.structure}.
Then, we describe the implementation of the level set function in Section~\ref{sec.levelset}. In 
Section~\ref{sec.locmodfe}, we give an overview on the additional steps needed compared to a standard
finite element code and how they are implemented in the class \texttt{LocModFE}. Finally, we show in 
Section~\ref{sec.UsingLocModFE}, how these are incorporated in a standard finite element program.

\subsection{Structure of the code}
\label{sec.structure}

The source code can be split into three parts, which can be found in the files
\texttt{locmodfe.h} and \texttt{.cc}, \texttt{step-modfe.cc} and \texttt{problem.h}. 
The following lines are copied from the preamble of the file \texttt{README.txt}:\\

\begin{lstlisting} 
 * The source code includes the following files and classes:
 *
 * 1) locmodfe.cc/h: Contain all functions that are specific to the locally 
 *			modified FE method
 *    a) class LocModFEValues : Extends the FEValues class in deal.II, where
 *			the local basis functions on the reference patches 
 *			are evaluated 
 *    b) class LocModFE : Key class of the locally modified finite element 
 *			method
 *
 * 2) step-modfe.cc: 
 *    a) class ParameterReader: Read in parameters from a seperate parameter 
 *			file
 *    b) class InterfaceProblem : local user file similar to many deal.II 
 *			tutorial steps, which controls the general workflow of
 *			the code, for example the solution algorithm, assembly
 *			of system matrix and right-hand side and output
 *    c) int main()
 *
 * 3) problem.h: Problem-specific definition of geometry, boundary conditions
 *			and analytical solution
 *    a) class LevelSet : Implicit definition of interface and sub-domains 
 *    b) classDirichletBoundaryConditions : Definition of the Dirichlet data
 *    c) class ManufacturedSolution : Analytical solution for error estimation
\end{lstlisting}

\subsubsection*{\texttt{locmodfe.h} and \texttt{locmodfe.cc}}
The files \texttt{locmodfe.h} and \texttt{locmodfe.cc} contain 
all functions that are specific for the locally modified finite element discretisation. The class \texttt{LocModFEValues} extends 
the \texttt{FEValues} class in \texttt{deal.II}. In this 
class the values of the basis functions and their gradients
(in deal.II ``shape functions'') as well as the derivatives of the map $\hat{T}_P$ are evaluated in quadrature points on the reference patch, 
depending on the reference patch type ($\hat{P}_0,...,\hat{P}_3$) and the boolean parameter \texttt{\_hierarchical}, which specifies if
a hierarchical basis is to be used.

In the class \texttt{LocModFE}, we check if patches are cut
and in which sub-domains they are (function \texttt{set\_material\_ids}), define the type of the cut (configurations A,...,D),  the reference patch type
($\hat{P}_0,...,\hat{P}_3$)
and the local mappings $\hat{T}_P$ (function \texttt{init\_FEM}).
Moreover, we initialise the respective quadrature formulas depending on the reference patches (function \texttt{compute\_quadrature}, 
more details on the quadrature will be given below), 
provide functions to compute norm errors (function \texttt{integrate\_difference\_norms}), to set Dirichlet boundary values in cut patches 
(function \texttt{interpolate\_boundary\_values})
and to visualise the solution (\texttt{plot\_vtk}).

\subsubsection*{\texttt{step-modfe.cc}}
In the file \texttt{step-modfe.cc}, we find the \texttt{main()} function and the classes
\texttt{ParameterReader} and \texttt{InterfaceProblem}.
The class \texttt{ParameterReader} is used to read in parameters from a parameter file, as in many \texttt{deal.II} tutorial steps.
The class \texttt{InterfaceProblem} can also be found similarly in many of the local user files in the tutorial steps. It contains 
for example the loops of the Newton iteration as well as functions to assemble the right-hand side and the system matrix. They differ
from other \texttt{deal.II} steps
only, when specific functions from the \texttt{LocModFE} class need to be used. The main modifications that are required for the locally 
modified finite element method will be explained in detail in the next section.

\subsubsection*{\texttt{problem.h}}
Finally, the file \texttt{problem.h} contains three classes, where the
geometry, 
the Dirichlet boundary data and the analytical solution 
for the specific example to be solved are specified.

\subsection{The Level set function}
\label{sec.levelset}
  
  In order to assign an element type to a patch, let us assume,
  that the interface is represented as zero-contour of a Level-Set function $\chi(x)$. In our examples, 
  the function $\chi(x)=\|x-x_m\|^2 - 0.25, x_m=(0,y_{\rm offset})$ 
  is specified by the following expressions in the class \texttt{LevelSet} in the file 
  \texttt{problem.h}:  
  \begin{lstlisting}
template <int dim> class LevelSet
{
 ...
 public:
 
  // Compute value of the LevelSet function in a point p
  double dist(const Point<dim> p) const
  {
    return p(0)*p(0) + (p(1)-_yoffset)*(p(1)-_yoffset) -0.5*0.5;
  }

  // Derivatives for Newton's method to find cut position
  double dist_x(const Point<dim> p) const
  {
    return 2.0*p(0);
  }
  double dist_y(const Point<dim> p) const
  {
    return 2.0*(p(1)-_yoffset);
  }

  //Determine domain affiliation of a point p
  int domain(const Point<dim> p) const
  {
    double di = dist(p);
    if (di>=0) return 1;
    else return -1;
  }
  ...
};
  \end{lstlisting}
  The function \texttt{double dist(...)} can be used to obtain the value of $\chi$ in a point p. Moreover, we 
  provide the derivatives \texttt{double dist\_x(...)} and \texttt{double dist\_y(...)}, which will be needed by a Newton
  method to find the position, at which the interface cuts an exterior edge (see Point 3 below). By means of the 
  function \texttt{int domain(...)}, we obtain the index of the sub-domain, in which $p$ lies.

  \begin{figure}[t]
  \centering
  \begin{picture}(0,0)%
\includegraphics{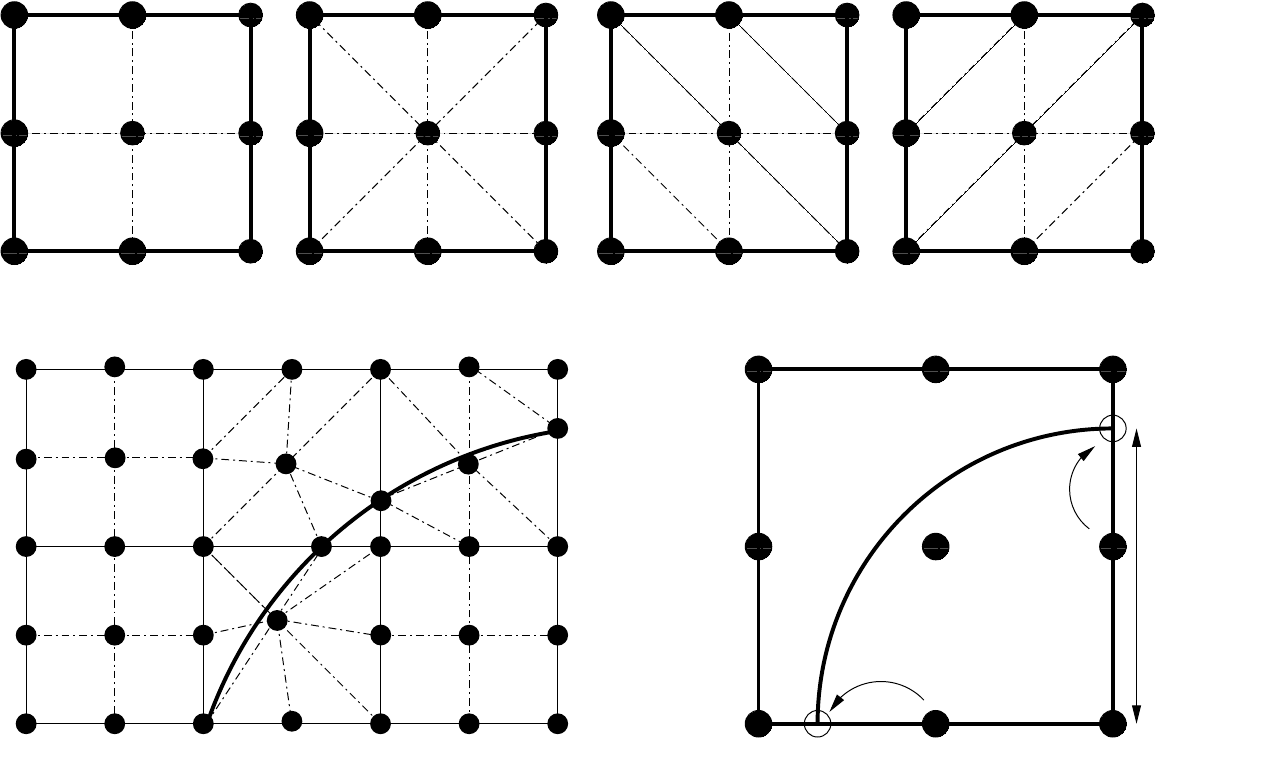}%
\end{picture}%
\setlength{\unitlength}{2486sp}%
\begingroup\makeatletter\ifx\SetFigFont\undefined%
\gdef\SetFigFont#1#2{%
  \fontsize{#1}{#2pt}%
  \selectfont}%
\fi\endgroup%
\begin{picture}(9709,5789)(1197,-5741)
\put(6076,-1636){\makebox(0,0)[lb]{\smash{{\SetFigFont{8}{9.6}{\color[rgb]{0,0,0}$\hat{P}_2$}%
}}}}
\put(7426,-3211){\makebox(0,0)[lb]{\smash{{\SetFigFont{8}{9.6}{\color[rgb]{0,0,0}$\chi<0$}%
}}}}
\put(8326,-5011){\makebox(0,0)[lb]{\smash{{\SetFigFont{8}{9.6}{\color[rgb]{0,0,0}$\chi>0$}%
}}}}
\put(9946,-4516){\makebox(0,0)[lb]{\smash{{\SetFigFont{8}{9.6}{\color[rgb]{0,0,0}$s$}%
}}}}
\put(8416,-3796){\makebox(0,0)[lb]{\smash{{\SetFigFont{8}{9.6}{\color[rgb]{0,0,0}$\psi=0$}%
}}}}
\put(10891,-4426){\makebox(0,0)[lb]{\smash{{\SetFigFont{8}{9.6}{\color[rgb]{0,0,0}$=0$}%
}}}}
\put(1531,-1636){\makebox(0,0)[lb]{\smash{{\SetFigFont{8}{9.6}{\color[rgb]{0,0,0}$\hat{P}_0$}%
}}}}
\put(3736,-1276){\makebox(0,0)[lb]{\smash{{\SetFigFont{8}{9.6}{\color[rgb]{0,0,0}$\hat{P}_1$}%
}}}}
\put(8281,-1276){\makebox(0,0)[lb]{\smash{{\SetFigFont{8}{9.6}{\color[rgb]{0,0,0}$\hat{P}_3$}%
}}}}
\put(9226,-2626){\makebox(0,0)[lb]{\smash{{\SetFigFont{8}{9.6}{\color[rgb]{0,0,0}$v_2$}%
}}}}
\put(9226,-5326){\makebox(0,0)[lb]{\smash{{\SetFigFont{8}{9.6}{\color[rgb]{0,0,0}$v_1$}%
}}}}
\put(10126,-4021){\makebox(0,0)[lb]{\smash{{\SetFigFont{8}{9.6}{\color[rgb]{0,0,0}$\chi(v_1+s(v_2-v_1))$}%
}}}}
\put(6571,-5686){\makebox(0,0)[lb]{\smash{{\SetFigFont{8}{9.6}{\color[rgb]{0,0,0}$-$}%
}}}}
\put(9811,-5686){\makebox(0,0)[lb]{\smash{{\SetFigFont{8}{9.6}{\color[rgb]{0,0,0}$+$}%
}}}}
\put(9811,-2716){\makebox(0,0)[lb]{\smash{{\SetFigFont{8}{9.6}{\color[rgb]{0,0,0}$-$}%
}}}}
\put(6571,-2671){\makebox(0,0)[lb]{\smash{{\SetFigFont{8}{9.6}{\color[rgb]{0,0,0}$-$}%
}}}}
\put(1576,-5281){\makebox(0,0)[lb]{\smash{{\SetFigFont{8}{9.6}{\color[rgb]{0,0,0}$\hat{P}_0$}%
}}}}
\put(1576,-3931){\makebox(0,0)[lb]{\smash{{\SetFigFont{8}{9.6}{\color[rgb]{0,0,0}$\hat{P}_0$}%
}}}}
\put(4231,-5281){\makebox(0,0)[lb]{\smash{{\SetFigFont{8}{9.6}{\color[rgb]{0,0,0}$\hat{P}_0$}%
}}}}
\put(3556,-5371){\makebox(0,0)[lb]{\smash{{\SetFigFont{8}{9.6}{\color[rgb]{0,0,0}$\hat{P}_1$}%
}}}}
\put(5176,-3661){\makebox(0,0)[lb]{\smash{{\SetFigFont{8}{9.6}{\color[rgb]{0,0,0}$\hat{P}_2$}%
}}}}
\put(2836,-3796){\makebox(0,0)[lb]{\smash{{\SetFigFont{8}{9.6}{\color[rgb]{0,0,0}$\hat{P}_3$}%
}}}}
\end{picture}%

    \caption{Implementation of the parametric patch-based
      approach. \textit{Top row:} Four different reference patches. \textit{Lower left:}
      Sample mesh with patches corresponding to all four variants. \textit{Lower right:} Identification of
      the cut points by means of the level set function $\chi$.}
    \label{fig:impl}
  \end{figure}
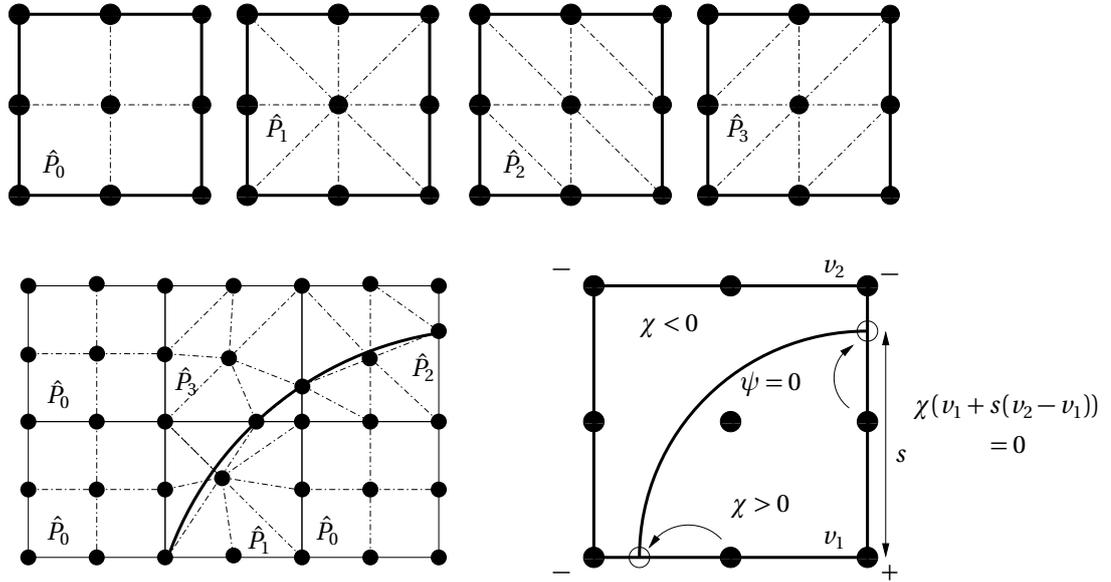

  \subsection{Implementation of the class \texttt{LocModFE}}
  \label{sec.locmodfe}
  
  Before we describe the additional steps needed for the locally modified finite element approach in detail,
  let us note that a patch is affected by the interface
  if $\chi$ shows different signs in two of the four outer vertices. 
  In the same way, we identify the edges cut by the interface. Let $v_1$
  and $v_2$ be the two outer nodes of an edge with
  $\chi(v_1)>0>\chi(v_2)$, see Figure~\ref{fig:impl}. The exact coordinate 
  where the interface line crosses an
  edge, can be found by a simple Newton method to find the zero $s_0$ of 
  \[
  f(s)=\chi\big(v_1+s( v_2- v_1)\big)=0.
  \]
  
  The following steps are executed in each patch $P\in \mathcal{T}_{2h}$ before the system matrix 
  and right-hand side are assembled. Note that all these operations are local operations on the patch level:
  \begin{enumerate}
   \item We equip the four exterior vertices $v_i, i=0,...,3$ of the patches with a colour (-1 or 1), based on the value of 
   the function \texttt{domain}($v_i$) in the LevelSet class.
   \item We equip each patch with a color (-1,0 or 1): -1 and 1 if the color of the four vertices in 1.$\,$is -1 or 1 for all of them, respectively;
   0 for interface patches with vertices in both sub-domains.
   \item If $P$ is an interface patch, find the two edges $e_1$ and $e_2$ affected by the interface by checking the color of the end vertices $v_1$ and $v_2$ as in 2.$\,$and compute 
   the exact cut position on both edges by using Newton's method to find the zeros $s_0$ of 
   \[
  f(s) = \chi\big(v_1+s( v_2- v_1)\big).
  \]
  \item Specify the type of the cut (configuration A,...,D) and define the reference patch type $\hat{P}_0,...,\hat{P}_3$.
  \item Define the local mapping $\hat{T}_P: \hat{P}_i \to P$ by means of the position of the 9 vertices in the physical patch $P$: 
        The degrees of freedom of the two edges $e_1$ and $e_2$ affected by the interface are moved to the point $v_1+s_0( v_2- v_1)$. The position of 
        the midpoint depends on the configuration $A,...,D$ (see Sections~\ref{sec:fe} and~\ref{sec_hierarchical}).
   \item Choose one of the four quadrature formulas, depending on the reference patch $\hat{P}_0,...,\hat{P}_3$.     
  \end{enumerate}

\subsubsection*{Step 1 and 2 (implemented in \texttt{set\_material\_ids})}

 We will provide now some code snippets to illustrate how these steps are implemented in the class \texttt{LocModFE}. Step 1 and 2 are implemented in the function
 \texttt{void set\_material\_ids}:
\begin{lstlisting}
template <int dim>
void LocModFE<dim>::set_material_ids (const DoFHandler<dim> &dof_handler,
			 const Triangulation<dim> &triangulation)
{
  ...
  unsigned int subdom1_counter;

  for (unsigned int cell_counter = 0; cell!=endc; ++cell, cell_counter++)
    {
      subdom1_counter = 0;

      for (unsigned int v=0; v<GeometryInfo<dim>::vertices_per_cell; ++v)
      {
        //First determine the sub-domain of the four outer vertices
        double chi_local = chi.domain(cell->vertex(v));
        node_colors[cell->vertex_index(v)] = chi_local;
       
        if (chi_local > 0) subdom1_counter ++;
      }
       
      //Based on the colors of the vertices, specify a color for the patches
      //(0 stands for an interface patch)
      if (subdom1_counter == 4)
        cell_colors[cell_counter] = 1;
      else if (subdom1_counter == 0)
        cell_colors[cell_counter] = -1;
      else 
        cell_colors[cell_counter] = 0;
       
    }
}
\end{lstlisting}

First, we set in line 16 the \texttt{node\_color} for each of the four outer vertices of the patch,
based on the value of the Level set function \texttt{chi} (step 1). Moreover, we count the number of 
outer vertices of the patch lying in sub-domain 1 (line 18) by means of the counter 
\texttt{subdom1\_counter}.
If the result is 0 or 4, the patch lies completely in one sub-domain and the \texttt{node\_color} of the four vertices 
(-1 or 1) is set as \texttt{cell\_color} for the patch; otherwise we set the \texttt{cell\_color} to 0, which 
corresponds to an interface patch (line 28).

\subsubsection*{Step 3 to 5 (implemented in \texttt{init\_FEM})}
The steps 3-5 can be found in the function 
\begin{lstlisting}
 void init_FEM(const typename DoFHandler<dim>::active_cell_iterator &cell,
      unsigned int cell_counter, FullMatrix<double> &M, 
      const unsigned int dofs_per_cell, unsigned int &femtype_int, 
      std::vector<double> &LocalDiscChi, std::vector<int>& NodesAtInterface);
\end{lstlisting}
As the implementation of this function is quite lengthy, let us only discuss its outputs:
The resulting reference patch type (step 4) is written to the variable \texttt{femtype\_int}. As 
shown in \eqref{TP}, the map 
$\hat{T}_P$ can be parametrised by the coordinates of the nine vertices $x_i^P, i=1,...,9$ in the physical patch $P$.
These are memorised in the $2\times 9$-matrix \texttt{M}. Moreover, we would like to mention the vector \texttt{LocalDiscChi}, 
which contains the nine values of the Level set function $\chi(x_i^P)$ in the vertices. These parametrise a discrete level set function 
$\chi_h$, that will be used in the computations, see the following paragraph.

\subsubsection*{Step 6 (implemented in \texttt{compute\_quadrature})}

For the choice of the quadrature formula depending on the reference patch type (step 6), we use the function
\begin{lstlisting}
Quadrature<dim> compute_quadrature (int femtype);
\end{lstlisting}
The four different quadrature formulas that can be chosen are defined in the function
\begin{lstlisting}
void initialize_quadrature();
\end{lstlisting}
that has to be called once in the beginning of the program (for example within the function \texttt{run}, see Section~\ref{sec.UsingLocModFE}). The integration
points are chosen as the four Gauss points
of the Gaussian integration formula of order one in each of the sub-quadrilaterals and as the three Gauss points 
of the corresponding Gaussian integration formula in each of the sub-triangles. This results in a total of 16 integration points in 
regular patches and of 24 integration points in interface patches.

\subsection{Using the functions of the class \texttt{LocModFE} in a standard finite element program}
\label{sec.UsingLocModFE}

In order to access the functions of the class \texttt{LocModFE} we have added the object
\begin{lstlisting}
 LocModFE<dim> lmfe;
\end{lstlisting}
as a member to the user class \texttt{InterfaceProblem}.

\subsubsection*{The \texttt{run} method}
As in {almost all} \texttt{deal.II} tutorial steps, the workflow of the code is controlled by the function \texttt{void run()}
of the user class \texttt{InterfaceProblem}. 
We show this function here for test case 2, skipping some lines 
with \texttt{'...'} that contain only output to the console (\texttt{std::cout <<}):

\begin{lstlisting}
template <int dim>
void InterfaceProblem<dim>::run ()
{
  set_runtime_parameters();
  setup_system();
  lmfe.initialize_quadrature();

  //Memorize initial solution  
  Vector<double> initial_solution = solution;
  std::cout << std::endl;
  
  if (test_case == 1)
    { 
     ...
    }
  else if (test_case == 2)
    {
      for (unsigned int i=0; i < N_testcase2; ++i)
      {
        // Move y-position of circle at each step
        _yoffset = (double)i / (double)N_testcase2 * min_cell_vertex_distance;
        lmfe.LevelSetFunction()->set_y_offset (_yoffset);

        // Reset material_ids based on the new interface location 
        lmfe.set_material_ids (dof_handler, triangulation);
        
        std::cout << ...
        
        // Solve system with Newton solver
        newton_iteration (); 
        
        // Compute functional values (error norms)
        compute_functional_values(false);
        
        ...
        // Write solutions as *.vtk file
        lmfe.plot_vtk (dof_handler,fe,solution,i); 
      }           
    } // end test_case 2
}

\end{lstlisting}

\subsubsection*{\texttt{void initialize\_quadrature()} and level set function}

The first function of the class \texttt{LocModFE} that is used, is \texttt{void initialize\_quadrature()} in line 6, which initialises
the four quadrature formulas for the four reference patch types $\hat{P}_0,...,\hat{P}_3$, that can be accessed by means of 
\texttt{lmfe.compute\_quadrature(int femtype)} later on. In the lines 21 and 22, the vertical position of the circular interface is updated by means 
of the y-coordinate (\texttt{\_yoffset}) of the midpoint $x_m$ of the circle and then passed to the level set function of the class \texttt{LocModFE}. 
Remember that in this test case the interface is moved gradually upwards.

\subsubsection*{\texttt{void set\_material\_ids(...)} and \texttt{newton\_iteration()}}

Next, the function \texttt{void set\_material\_ids(...)} is called in line 25, which sets the colours for vertices and patches as explained above.
All the computations are then done within the function \texttt{newton\_iteration()} in line 30. The source code of this function itself contains no
content that is specific to the locally modified finite element method.
In fact the Newton solver is mostly copy and paste from \cite{Wi11_fsi_with_deal}.
The only modified functions that are called within \texttt{newton\_iteration()}
are the assembly of the system matrix and right-hand side, which will be discussed below and 
the function \texttt{set\_initial\_bc}, which has to be modified
in interface patches by calling\\ \texttt{lmfe.interpolate\_boundary\_values(...)}. After the 
Newton iteration, functional values are computed in 
the function \texttt{compute\_functional\_values(...)} that uses 
the modified function \texttt{lmfe.integrate\_difference\_norms(...)}. Finally, the results are written to a vtk 
file by \texttt{lmfe.plot\_vtk} in line 37, together with a mesh consisting of the sub-cells of the patches.

\subsubsection*{\texttt{assemble\_system\_matrix()}}

Within the function \texttt{newton\_iteration}, the functions \texttt{assemble\_system\_matrix()} and
\texttt{assemble\_system\_rhs()}
are called. We show here the prior exemplarily, the modifications in the assembly of the right-hand side are analogous:
\begin{lstlisting}
template <int dim>
void InterfaceProblem<dim>::assemble_system_matrix ()
{
  ...
  LocModFEValues<dim>* fe_values;

  //We initialize one LocModFEValue object for patch type 0 and one for patch 
  //types 1 to 3, due to the different number of integration points
  Quadrature<dim> quadrature_formula0 = lmfe.compute_quadrature(0);
  LocModFEValues<dim> fe_values0 (fe, quadrature_formula0,
	      _hierarchical, update_values | update_quadrature_points  |
	      update_JxW_values | update_gradients);

  Quadrature<dim> quadrature_formula1 = lmfe.compute_quadrature(1);
  LocModFEValues<dim> fe_values1 (fe, quadrature_formula1,
	      _hierarchical, update_values | update_quadrature_points  |
	      update_JxW_values | update_gradients);       
\end{lstlisting}

After some variable definitions that we have skipped here in line 4, we initialise a pointer\\ \texttt{LocModFEValues<dim>* fe\_values(...)}.
Depending on the patch type, this pointer will be set for each patch in the following loop to one of the objects \texttt{LocModFEValues<dim> fe\_values0(...)} 
(patch type $\hat{P}_0$) or 
\texttt{LocModFEValues<dim> fe\_values1} (patch type $\hat{P}_1,...,\hat{P}_3$) defined in the lines 11 and 16. We initialise
these two objects before the loop over all patches for efficiency reasons. Two different objects are needed as the local number of quadrature points is different 
for patch type $\hat{P}_0$ compared to the interface patch types.

Next, we start the loop over all patches, in which the local contribution to the global system matrix is computed.
Before we can compute the local basis functions and their gradients, we have to call the function 
\texttt{init\_FEM(...)}, that sets the patch type (\texttt{femtype}), 
the local mapping $\hat{T}_P$ (\texttt{M}) and the discrete level set function $\chi_h$ (\texttt{LocalDiscChi}), see line 25. Then, 
the quadrature formula that corresponds to the patch type is set in line 27 and one of the two objects 
of type \texttt{LocModFEValues}, that were initialised above, is chosen. The quadrature formula, as well as the patch type and the local mapping $\hat{T}_P$
are then passed to this object in line 34. Now, we are ready to compute the local basis functions, their gradients and the derivatives 
of the mapping $\hat{T}_P$, that are needed to compute the entries of the system matrix. This is 
done by \texttt{fe\_values->reinit(J)}
in line 38:\\

\begin{lstlisting}[firstnumber=19]
  for (unsigned int cell_counter = 0; cell!=endc; ++cell,++cell_counter)
    {
      local_matrix=0;
      
      //Set patch type (femtype), map T_P (M), local level set function (LocalDiscChi)
      //and list of nodes at the interface
      lmfe.init_FEM (cell,cell_counter,M,dofs_per_cell,femtype, LocalDiscChi, NodesAtInterface);

      Quadrature<dim> quadrature_formula = lmfe.compute_quadrature(femtype);
      const unsigned int   n_q_points      = quadrature_formula.size();

      //Choose one of the initialized objects for LocModFEValues      
      if (femtype==0) fe_values = &fe_values0;
      else fe_values = &fe_values1;

      fe_values->SetFemtypeAndQuadrature(quadrature_formula, femtype, M);

      std::vector<double> J(n_q_points); 
      //Now the shape functions on the reference patch are initialized
      fe_values->reinit(J);
\end{lstlisting}

Next, we have a loop over the quadrature points and over the local degrees of freedom as usual in a finite element program. In order 
to compute the diffusion coefficient $\kappa$ (\texttt{viscosity}), we use the discrete level set function $\chi_h$. The value of $\chi_h$
in the quadrature point $q$ is extracted from the vector \texttt{LocalDiscChi} by the function \texttt{lmfe.ComputeLocalDiscChi(...)} in line 50.
Note that it is important to use this discrete level set function for the assembly of matrix and right-hand side, as otherwise $\kappa$
would jump within a sub-element and the program would not 
be robust with respect to high-contrast coefficients. The remaining 
lines are {standard and very similar to many other \texttt{deal.II} tutorial
steps (e.g., deal.II-step-22)}:\\ 

\begin{lstlisting}[firstnumber=39]
      for (unsigned int q=0; q<n_q_points; ++q)
      {
         for (unsigned int k=0; k<dofs_per_cell; ++k)
          {
            phi_i_u[k]       = fe_values->shape_value (k, q);
            phi_i_grads_u[k] = fe_values->shape_grad (k, q);
          }
          
          //Get the domain affiliation to set the viscosity. 
          //This is based on the discrete level set function, such that
          //all quadrature points in a sub-cell lie in the same sub-domain 
          lmfe.ComputeLocalDiscChi(ChiValue, q, *fe_values, dofs_per_cell, LocalDiscChi);

          if (ChiValue < 0.)  viscosity = visc_1; //Subdomain Omega_1 (inside the circle) 
          else viscosity = visc_2; //Subdomain Omega_2 (outside the circle)
                
          //Compute matrix entries as in other deal.II program.
          for (unsigned int i=0; i<dofs_per_cell; ++i)
            {
            for (unsigned int j=0; j<dofs_per_cell; ++j)
              {
                local_matrix(j,i) += viscosity *
                    phi_i_grads_u[i] * phi_i_grads_u[j] * J[q] * quadrature_formula.weight(q);
              }
            }
      }
      
      //Write into global matrix
      ...
    }
    ...
}
\end{lstlisting}

Finally, we remark that besides the described function calls
no further modifications are necessary in comparison to any other standard 
FEM code or deal.II tutorial program.

\section{Conclusion and outlook}
\label{sec.conclusion}

In this paper, we have explained the implementation of the locally fitted finite
element method first proposed in \cite{FreiRichter2014} in detail. The underlying 
framework is based on the open-source finite element library deal.II, 
\cite{dealII85}.
Moreover, we have illustrated the performance of the method by means of 
two numerical tests. We have shown that iterative methods such as the 
CG method can be used to solve the arising linear systems of equations and 
analysed the performance of 
the linear iterative solvers with respect to mesh refinement and different 
anisotropies.

The method can be applied to simulate the Stokes or Navier-Stokes equations
with equal-order elements and pressure stabilisations. The only difficulty lies in the treatment of the 
anisotropic cells within the stabilisation terms. A solution for the Continuous Interior Penalty (CIP) stabilisation
has been proposed by \cite{FreiDiss}, \cite{FreiPressureStab}.
%In ongoing work, we plan to implement an \textit{inf-sup} version of the locally modified
%finite element scheme for based on $P_2$-$P_1$-Taylor-Hood elements.

In order to obtain higher-order accuracy, the interface has to be resolved with higher order. This
can be achieved by using maps $\hat{T}_P$ of higher polynomial degrees. We would like to remark, however, that
this might lead to additional difficulties concerning the degeneration of the sub-elements
within the patches. A promising alternative is the use of so-called ``boundary value correction'' techniques at
the interface, see \cite{Burmanetal2018}.

% Concerning discretisation in time for non-stationary problems with moving interfaces, one has to take 
% into account that typically the solutions of interface problems show similar regularity issues in time 
% across the interface as in space (see the discussion in the introduction). Possible solutions include 
% space-time discretisation techniques (\cite{}Langer, BehrTezduyar), the Characteristics/Galerkin approach 
% (Pironneau), modified Galerkin time-stepping techniques (\cite{FreiRichter}) or time-stepping schemes 
% based on discontinuous Galerkin approaches of order $r\geq1$ in time(\cite{}Lehrenfeld, Zahedi).

% An analogous construction of a locally modified finite element scheme is possible in three dimensions. 
% Patch elements consisting of tetra- or hexahedra can be split into sub-tetrahedra, in order 
% to resolve arbitrary cuts of the patches. As in two dimensions, a maximum angle condition needs to be fulfilled
% in each sub-element in order to obtain error estimates that are robust with respect to the position of 
% the interface. As for the hierarchical approach presented in this work, it might be necessary to move certain
% exterior patch vertices.

%{ 3d? Das fragen die Reviewer sonst garantiert...}

Moreover, the locally modified FEM method has a natural extension to three space dimensions. The
mathematical, numerical, and algorithmic requirements are currently ongoing work.
Another desirable feature is the parallelisation of the
approach. Here, we do not assume major difficulties since the programming structure 
is similar to step-42 of the deal.II tutorial programs. As all 
the additions compared to a standard deal.II code are local on the patch level, this should 
in principle be possible without further difficulties.

%{ Concerning the implementation it would be further desirable to parallelise the approach. As all 
%the additions compared to a standard deal.ii code are local on the patch level, this should 
%in principle be possible without further difficulties. }

% {TW: Ich wuerde vielleicht nicht zu viele future extension auffuehren:
%   das hat zwei negative Effekte: 1. hat man als reviewer den Eindruck das wir 
%   nichts neues gemacht haben, wenn es so viele offene Fragen gibt und zweitens
%   muessen wir den Lesern ja nicht alle Moeglichkeiten zur Erweiterung
%   preisgeben und machen die wirklichen spannenden Dinge andere vor uns
%   ... Deshalb ich die nachfolgenden beiden Absaetze auskommentiert.}

%Finally, we plan to use the software for complex multi-physics problems, such as fluid-structure interactions
%or multi-phase flow.

%%%%%%%%%%%%%%%%%%%%%%%%%%%%%%%%%%%%%%%%%%%%%%%%%%%%%%%%%%%%%%%%%%%%%%%%%%%%%%%%%%
\section{Acknowledgements}
The first author was supported 
 by the DFG Research Scholarship FR3935/1-1.
The third author gratefully acknowledges the travel support from
University College London (i.e., Eric Burman) for finalising this work.

%%%%%%%%%%%%%%%%%%%%%%%%%%%%%%%%%%%%%%%%%%%%%%%%%%%%%%%%%%%%%%%%%%%%%%%%%%%
\bibliographystyle{plainnat}

%\bibliography{submitting}

\end{document}